\documentclass[10pt,reqno]{amsart}
\usepackage{amssymb}
\usepackage{amsmath}
\usepackage{amsthm}
\usepackage{graphicx}

\usepackage{subfigure}

\usepackage{caption}
\usepackage{bm}
\usepackage[colorlinks=true,urlcolor=blue,
citecolor=red,linkcolor=blue,linktocpage,pdfpagelabels,
bookmarksnumbered,bookmarksopen]{hyperref}
\usepackage[titletoc,title]{appendix}
\numberwithin{equation}{section} 
\newcounter{cont}[section] 

\newtheorem{thm}[cont]{Theorem}
\newtheorem{prop}[cont]{Proposition}
\newtheorem{lem}[cont]{Lemma}

\theoremstyle{definition}
\newtheorem{defn}[cont]{Definition}
 \theoremstyle{remark}
 \newtheorem{rem}[cont]{Remark}

\newcommand{\R}{\mathbb{R}}
\newcommand{\e}{\varepsilon}

\begin{document}

\title[Metastable patterns for mean curvature reaction-diffusion models]{Metastable patterns for a reaction-diffusion model with mean curvature-type diffusion}

\author[R. Folino]{Raffaele Folino}

\address[R. Folino]{Departamento de Matem\'aticas y Mec\'anica\\Instituto de 
Investigaciones en Matem\'aticas Aplicadas y en Sistemas\\Universidad Nacional Aut\'onoma de 
M\'exico\\Circuito Escolar s/n, Ciudad Universitaria C.P. 04510 Cd. Mx. (Mexico)}

\email{folino@mym.iimas.unam.mx}

\author[R.G. Plaza]{Ram\'on G. Plaza}

\address[R. G. Plaza]{Departamento de Matem\'aticas y Mec\'anica\\Instituto de 
Investigaciones en Matem\'aticas Aplicadas y en Sistemas\\Universidad Nacional Aut\'onoma de 
M\'exico\\Circuito Escolar s/n, Ciudad Universitaria C.P. 04510 Cd. Mx. (Mexico)}

\email{plaza@mym.iimas.unam.mx}
	 
\author[M. Strani]{Marta Strani}

\address[M. Strani]{Dipartimento di Scienze Molecolari e Nanosistemi\\Universit\`a Ca' Foscari Venezia Mestre\\Campus Scientifico\\Via Torino 155, 30170 Venezia Mestre (Italy)}

\email{marta.strani@unive.it}



\maketitle


\begin{abstract} 
Reaction-diffusion equations are widely used to describe a variety of phenomena such as pattern formation
and front propagation in biological, chemical and physical systems. In the one-dimensional model with a balanced bistable reaction function, it is well-known that there is persistence of metastable patterns for an exponentially long time,
i.e. a time proportional to $\exp(C/\e)$ where $C,\e$ are strictly positive constants and $\e^2$ is the diffusion coefficient.
In this paper, we extend such results to the case when the linear diffusion flux is substituted by the mean curvature operator both in Euclidean and Lorentz--Minkowski spaces.
More precisely, for both models, we prove existence of metastable states which maintain a transition layer structure for an exponentially long time and we show that the speed of the layers is exponentially small.
Numerical simulations, which confirm the analytical results, are also provided.
\end{abstract}


\section{Introduction}\label{sec:intro}
In this paper we study the long time dynamics of the solutions to the reaction-diffusion equation 
\begin{equation}\label{eq:Q-model}
	u_t=Q(\e^2u_x)_x-F'(u),
\end{equation}
where $u=u(x,t) : [a,b]\times(0,+\infty)\rightarrow \R$, $\e>0$ is a small parameter 
and $F:\R\to\R$ is a double well potential with wells of equal depth;
more precisely, we assume that the potential $F\in C^3(\R)$ satisfies 
\begin{equation}\label{eq:ass-F}
	F(\pm1)=F'(\pm1)=0, \qquad F''(\pm1)>0, \qquad F(u)>0 \quad \forall\, u\neq\pm1.
\end{equation}
We consider equation \eqref{eq:Q-model} complemented with homogeneous Neumann boundary conditions
\begin{equation}\label{eq:Neu}
	u_x(a,t)=u_x(b,t)=0, \qquad \qquad t>0,
\end{equation}
and initial datum
\begin{equation}\label{eq:initial}
	u(x,0)=u_0(x), \qquad \qquad x\in[a,b].
\end{equation}
Regarding the diffusion flux $Q$, we will consider two different choices: 
the first one is
\begin{equation}\label{eq:curv+}
	Q(s):=\frac{s}{\sqrt{1+s^2}},
\end{equation}
 corresponding to the \emph{mean curvature operator in Euclidean space}, while the second one is
\begin{equation}\label{eq:curv-}
	Q(s):=\frac{s}{\sqrt{1-s^2}},
\end{equation}
and corresponds to the \emph{mean curvature operator in Lorentz--Minkowski space}.

From a physical point of view, the choice of the saturating diffusion \eqref{eq:curv+} was introduced by Rosenau \cite{Rsn90}:
in the theory of phase transitions, the author extends the Ginzburg--Landau free-energy functionals 
and considers a free-energy functional, which has a linear growth rate with respect to the gradient, in order to include interaction due to high gradients. 
Next, the effects of saturating diffusion on convection and reaction processes have been studied in many papers,
among others see \cite{FGS,GaSa15,GaSt19b,KuRo97,KuRo06} and references therein.
We also recall the study of the interaction between flux-saturation effects with those of porous media flow \cite{CCCSS16,CCCSS17}.

Regarding the unbounded and singular function \eqref{eq:curv-}, the nonlinear differential operator  
\begin{equation}\label{mink}
	\textrm{div}\left(\frac{\nabla u}{\sqrt{1-|\nabla u|^2}}\right)
\end{equation}
appears in many applications and it is usually meant as \emph{mean curvature operator in the Lorentz--Minkowski space}.
It is of interest in differential geometry, general relativity and appears in the nonlinear theory of electromagnetism,
where it is usually referred to as Born-Infeld operator; for details see, among others, \cite{BaSi82}, \cite{BAPo16} and references therein.

The difference with respect to the saturating diffusion operator is that, in the case of \eqref{mink}, 
the free-energy functional behaves quadratically for small values of the gradient, 
and it reaches its extremal value with infinite slope when $|\nabla u|$ reaches its upper bound.
It follows that the diffusion flux $\displaystyle \frac{\nabla u}{\sqrt{1-|\nabla u|^2}}$ is singular at the
maximal value of the gradient.

From a mathematical point of view, the main difference between the functions \eqref{eq:curv+} and \eqref{eq:curv-}
is that the first one is a bounded function satisfying 
\begin{equation*}
	\lim_{s\to\pm\infty} Q(s)=\pm1, \qquad \qquad \lim_{s\to\pm\infty} Q'(s)=0,
\end{equation*}
while the second one is an unbounded and singular function satisfying
\begin{equation*}
	\lim_{s\to\pm1} Q(s)=\pm\infty, \qquad \qquad \lim_{s\to\pm1} Q'(s)=+\infty, \qquad \qquad Q'(s)\geq1, \qquad \forall\,s\in\R.
\end{equation*}
Such a difference notably changes the character of the PDE \eqref{eq:Q-model}, as we describe in the sequel.
First of all, notice that in the case of smooth solutions, we can expand the term $Q(\e^2u_x)_x$ and rewrite equation \eqref{eq:Q-model} as
\begin{equation}\label{eq:Q-model2}
	u_t=\e^2Q'(\e^2u_x)u_{xx}-F'(u),
\end{equation}
where $Q'(s)={\left(1\pm s^2\right)^{-3/2}}$.
Therefore, in the case \eqref{eq:curv+} the \emph{diffusion coefficient} is positive but vanishes when $|u_x|\to\infty$,
while in the case \eqref{eq:curv-} it is strictly positive if $\e^2|u_x|<1$.
As a consequence, when the flux function $Q$ is bounded and $Q'$ vanishes at $\pm\infty$, 
the solutions of \eqref{eq:Q-model} can exhibit \emph{hyperbolic phenomena},
such as formation of discontinuities also if the initial datum is smooth (see \cite{KuRo06}).
In the particular case \eqref{eq:curv+} we have a \emph{strongly degenerate parabolic} equation,
as it is indicated in \cite{BerDal92}, where the authors study the Cauchy problem for equation \eqref{eq:Q-model} with $F=0$
and show that the condition
\begin{equation*}
	\int_0^\infty s\,Q'(s)\,ds<\infty
\end{equation*}
implies persistence of discontinuous solution; more precisely, they show that if the initial datum is discontinuous
then there exists a finite time $T>0$ such that the solution remains discontinuous until time $t = T$ and becomes continuous thereafter.
For this reason, the case \eqref{eq:curv+} is also referred to as \emph{strong saturating diffusion}.

As it was previously mentioned, the case of equation \eqref{eq:Q-model} with a saturating diffusion 
and a double well potential $F$ with wells of equal depth has been studied in \cite{KuRo06}, 
where the authors prove that if the potential $F$ is sufficiently large, then there exist discontinuous equilibria; 
also, they numerically show that such discontinuous equilibria are attractive for a wide class of initial data
(including smooth initial data).
The need of a smallness condition on the potential $F$ in order to ensure the existence of smooth standing waves is confirmed in Proposition \ref{prop:ex-met},
where we consider equation \eqref{eq:Q-model} in $\R$ with $Q$ given by \eqref{eq:curv+} and we prove that there exist smooth stationary solutions 
connecting the two minimal points $\pm1$ if $F$ is sufficiently small, see condition \eqref{eq:maxF}.

Regarding the unbounded case  \eqref{eq:curv-}, we mention the paper \cite{BCNy17}, 
where the authors consider the free energy functional 
\begin{equation*}
	\mathcal{J}(u):=\int_{\R\times\omega}\left(1-\sqrt{1-|\nabla u|^2}+F(u)\right)\,dx,
\end{equation*}
with $\omega\in\R^{d-1}$, $d\geq1$ and $F$ satisfying \eqref{eq:ass-F}.
They look for minimizers of $\mathcal{J}$ in the space
\begin{align*}
	\mathcal{X}:=\Big\{u\in W^{1,\infty}(\R\times\omega)\, : \, \|\nabla u\|_{{}_{L^\infty}} & <1 \mbox{ and } \\
	 & \lim_{x\to\pm\infty} u(x,y)=\pm1 \mbox{ uniformly in } y\in\omega\Big\},
\end{align*}
and prove that any minimizer depends only on the first variable and is the unique solution (up to translations) to the boundary value problem
\begin{equation*}	
	\left(\frac{u'}{\sqrt{1-|u'|^2}}\right)'-F'(u)=0, \qquad \qquad \lim_{x\to\pm\infty}u(x)=\pm1.
\end{equation*}
As we will see in Section \ref{sec:met-}, the existence of such solution is crucial in our work and it suggests 
that there exist metastable patterns for equation \eqref{eq:Q-model} when $Q$ is given by \eqref{eq:curv-}. 

To conclude the discussion, we briefly mention that the evolution equation associated to \eqref{eq:curv-} 
has been recently studied in the presence of a convection term in \cite{FoSt}.

Motivated by previous results obtained for the solutions of reaction-diffusion equations with a linear diffusion flux, 
the main goal of this paper is to investigate the phenomenon of \emph{metastability} for equation \eqref{eq:Q-model}; 
from a general point of view, a metastable dynamics appears when the solution to an evolution PDE evolves very slowly in time 
and maintains an \emph{unstable} structure for a very long time.
Roughly speaking, the solution evolves so slowly that it \emph{appears} to be stable,
and after a very long time it undergoes a drastic change and converges to its asymptotic limit.
We refer to this type of solutions as \emph{metastable states} for the evolution PDE.

When choosing $Q(s)=s$ in \eqref{eq:Q-model} and $F$ satisfying \eqref{eq:ass-F}, one obtains the classical \emph{Allen--Cahn equation}
\begin{equation}\label{eq:Al-Ca}
	u_t=\e^2 u_{xx} -F'(u),
\end{equation} 
and it is well known that its solutions exhibit the phenomenon of metastability. 
More precisely, when considering equation \eqref{eq:Al-Ca} with homogeneous Neumann boundary conditions \eqref{eq:Neu},
the solutions maintain a \emph{transition layer structure} for an exponentially long time when the diffusion coefficient is small,
i.e. a time of order $\mathcal{O}(e^{C/\e})$ as $\e\to0$, and after that they converge to a constant equilibrium solution which is one of the minimal points of the potential $F$.

Many papers have been devoted to the study of the metastable dynamics for the Allen--Cahn equation \eqref{eq:Al-Ca};
being impossible to quote all the contributions, we mention here only the pioneering works \cite{BrKo90,CaPe89,CaPe90,FuHa89}.
It is to be observed that two main different approaches to metastability for equation \eqref{eq:Al-Ca} have been proposed in the aforementioned papers.
Since the method we use here is the \emph{energy approach} introduced by Bronsard and Kohn in \cite{BrKo90},
we briefly recall such strategy. The former is based on the fact that the Allen--Cahn equation can be seen as a gradient flow in $L^2(a,b)$ for the Lyapunov functional
\begin{equation*}
	\mathcal{F}[u]=\int_a^b\left[\frac{\e^2}2u^2_x+F(u)\right]\,dx.
\end{equation*}
In \cite{BrKo90}, the authors consider the case $F(u)=\frac14(u^2-1)^2$ and make use of the notion of $\Gamma$-convergence (\cite{Mod87,Ster88}) for the functional
\begin{equation*}
	\mathcal{F}_\e[u]=\int_a^b\left[\frac\e2u^2_x+\frac{F(u)}\e\right]\,dx,
\end{equation*}
which is obtained multiplying by $\e^{-1}$ the functional $\mathcal{F}[u]$.
In this way, Young inequality gives 
\begin{equation*}
	\mathcal{F}_\e[u]\geq \int_a^b|u_x|\sqrt{2F(u)}\,dx=\int_{-1}^{+1}\sqrt{2F(s)}\,ds=:c_0,
\end{equation*}
for any function $u$ connecting $-1$ and $+1$.
The positive constant $c_0$ (independent on $\e$ for the renormalization) represents the minimum energy
to have a transition between $-1$ and $+1$, and one can prove that if $u^\e$ converges in $L^1(a,b)$ to a step function $v$ which
takes only the values $\pm1$ and has exactly $N$ jumps, then
\begin{equation*}
	\liminf_{\e\to0^+} \mathcal{F}_\e[u^\e]\geq Nc_0.
\end{equation*}
Moreover, it is possible to choose a function such that the equality holds.
Using such properties of the functional $\mathcal{F}_\e$, in \cite{BrKo90} the authors show persistence of unstable structures
with $N$ transitions between $-1$ and $+1$ for a time $T_\e=\mathcal{O}(\e^{-k})$, for any $k\in\mathbb{N}$.

The energy approach of \cite{BrKo90} is very powerful and permits to prove existence of metastable states for different evolution PDEs.
For instance, it was improved and used  in \cite{Grnt95} to prove persistence for an exponentially long time of the metastable states for the Cahn--Morral system,
and it was also applied to  the hyperbolic Cahn--Hilliard equation in \cite{FLM19} (see also \cite{Fol17,FLM17} for further information on application of the energy approach to metastability analysis of hyperbolic equations).

Following the same ideas, in this paper we modify the energy approach to rigorously establish the existence of metastable states for the IBVP \eqref{eq:Q-model}-\eqref{eq:Neu}-\eqref{eq:initial}. The main goal of this paper is to prove that if the initial datum $u_0$ depends on $\e$ and has an $N$-transition layer structure (see Definition \ref{def:TLS} below) then the solution to \eqref{eq:Q-model}-\eqref{eq:Neu}-\eqref{eq:initial} maintains the same $N$-transition layer structure for an exponentially long time. In this fashion, we show that the phenomenon of \emph{metastability} is still present in reaction-diffusion problems 
where the diffusion is nonlinear and it is given by either the strongly saturating diffusion function \eqref{eq:curv+}, or by the mean curvature (in the Lorentz-Minkowski space) unbounded diffusion function \eqref{eq:curv-}. 
For that purpose, it is to be noticed that in the case of a generic reaction-diffusion equation of the form \eqref{eq:Q-model}, the energy functional is given by
\begin{equation*}
	E[u]=\int_a^b\left[\tilde{Q}(u_x)+F(u)\right]\,dx,
\end{equation*}
where $\tilde{Q}'(s)=Q(s)$. Indeed, one can prove (see Lemma \ref{lem:energy}) that if $u$ is a solution to \eqref{eq:Q-model} with homogeneous Neumann boundary conditions \eqref{eq:Neu}, then 
\begin{equation}\label{eq:energy-var-intro}
	\frac{d}{dt} E[u](t)=-\int_a^b u_t^2(x,t)\,dx.
\end{equation}
In the case of the choices \eqref{eq:curv+}-\eqref{eq:curv-} for the diffusion function, we clearly have
\begin{equation*}
	\tilde{Q}(s)=c\pm\frac{\sqrt{1\pm\e^4s^2}}{\e^2}, \qquad \qquad c\in\R,
\end{equation*}
where the sign $+$  and $-$ corresponds to \eqref{eq:curv+}  and to \eqref{eq:curv-}, respectively.

In order to apply the same strategy of \cite{BrKo90}, we choose $\tilde{Q}(0)=0$ and use the renormalized energies,
\begin{equation}\label{eq:energy+}
	E_\e[u]=\int_a^b\left[\frac{\sqrt{1+\e^4u_x^2}-1}{\e^3}+\frac{F(u)}\e\right]\,dx,
\end{equation}
in the case where $Q$ is given by \eqref{eq:curv+}, and 
\begin{equation}\label{eq:energy-}
	\mathcal{E}_\e[u]=\int_a^b\left[\frac{1-\sqrt{1-\e^4u_x^2}}{\e^3}+\frac{F(u)}\e\right]\,dx,
\end{equation}
if $Q$ is given by \eqref{eq:curv-}.

The goal is to prove that, as in the \emph{linear} case, both energies $E_\e[u]$ and $\mathcal{E}_\e[u]$ satisfy the following property:
if $u^\e$ converges in $L^1(a,b)$ to a step function $v$,
taking only the values $\pm1$ and having exactly $N$ jumps, then
\begin{equation*}
	\liminf_{\e\to0^+} E_\e[u^\e]\geq Nc_0, \qquad \mbox{ and } \qquad \liminf_{\e\to0^+} \mathcal{E}_\e[u^\e]\geq Nc_0, 
\end{equation*}
with the equality holding for a particular choice of $u^\e$.
More precisely, we will prove that for any function $u$ sufficiently close in $L^1(a,b)$ to $v$, the following inequality holds
\begin{equation*}
	E_\varepsilon[u]\geq Nc_\e-C\exp(-A/\varepsilon),
\end{equation*}
for some $c_\e>0$ which converges  to $c_0$ as $\e\to0$ (see Proposition \ref{prop:lower}). A similar inequality holds for $\mathcal{E}_\e[u]$, see Proposition \ref{prop:lower-} below.
Such lower bounds are the key points to prove the persistence of metastable patterns for \eqref{eq:Q-model} in the cases \eqref{eq:curv+}-\eqref{eq:curv-}; in order to prove them we use two different inequalities (see \eqref{eq:strangeineq+} and \eqref{eq:strangeineq-}), 
which play the same role of Young inequality in the linear case \eqref{eq:Al-Ca}.
We underline that such results on the energy functionals, $E_\e[u]$ and $\mathcal{E}_\e[u]$, hold for a generic function $u\in H^1(a,b)$ and that
equation \eqref{eq:Q-model} plays no role.
Thanks to these variational results and equality \eqref{eq:energy-var-intro} we are able to apply the energy approach of Bronsard and Kohn \cite{BrKo90},
and by using the same procedure we prove the \emph{exponentially slow motion} of the solutions to \eqref{eq:Q-model}-\eqref{eq:Neu}. 

We close this Introduction by sketching a short plan of the paper; 
in Sections \ref{sec:met+} and \ref{sec:met-} we exploit the strategy described above for the two choices of diffusion fluxes \eqref{eq:curv+}  and \eqref{eq:curv-} under consideration, and prove that the solutions to \eqref{eq:Q-model} indeed exhibit a metastable dynamics. 
In particular, after showing the existence of metastable states having an $N$-transition layer structure (for more details see Definition \ref{def:TLS}), 
in Theorems \ref{thm:main} and \ref{thm:main-} we prove their persistence for an exponentially long time interval of the order $e^{C/\e}$, $C>0$.
Moreover, we provide an estimate for the speed of the transition layers, showing that it is (at most) exponentially small when $\e\to0$, see Theorem \ref{thm:interface} below.
Such results extend to the case of nonlinear diffusions as the ones in \eqref{eq:curv+}  and \eqref{eq:curv-} the results previously obtained for the standard Allen-Cahn equation.
In Section \ref{sec:num} we give numerical evidences of the rigorous results obtained in the Sections \ref{sec:met+}-\ref{sec:met-},
and we also show that metastability is a general phenomenon of equation \eqref{eq:Q-model},
that is, the metastable states are attractive for a large class of initial data.
Finally, we numerically show that the assumptions \eqref{eq:ass-F} on the potential $F$ are necessary for the appearance of metastability, as it happens in the linear case.

\section{Slow motion in the case of a saturating diffusion}\label{sec:met+}
The aim of this section is to show that there exist metastable states for the IBVP \eqref{eq:Q-model}-\eqref{eq:Neu}-\eqref{eq:initial}, 
when the function $Q$ is given by \eqref{eq:curv+}, and that such metastable states maintain the same unstable structure of the initial datum for an exponentially long time,
i.e. for a time $T_\e\geq e^{C/\e}$, with $C>0$ independent of $\e$.
Before stating the main result of this section, we present two preliminary lemmata, which are fundamental to develop
the energy approach mentioned in the Introduction.
The first lemma shows that the energy \eqref{eq:energy+} is a non-increasing function of time $t$ 
along the solutions of \eqref{eq:Q-model}-\eqref{eq:curv+} with homogeneous Neumann boundary conditions \eqref{eq:Neu}.
\begin{lem}\label{lem:energy}
Let $u\in C([0,T],H^2(a,b))$ be solution of \eqref{eq:Q-model}-\eqref{eq:Neu}, with $Q$ as in \eqref{eq:curv+}.
If $E_\e$ is the functional defined in \eqref{eq:energy+}, then 
\begin{equation}\label{eq:energy-dec}
	\frac{d}{dt}E_\e[u](t)=-\e^{-1}\|u_t(\cdot,t)\|^2_{{}_{L^2}},
\end{equation}
for any $t\in(0,T)$.
\end{lem}
\begin{proof}
By differentiating with respect to time \eqref{eq:energy+}, we obtain 
\begin{equation*}
	\frac{d}{dt}E_\e[u]=\frac1\e\int_a^b\left[\frac{\e^2 u_x}{\sqrt{1+\e^4u_x^2}}u_{xt}+F'(u)u_t\right]\,dx,
\end{equation*}
Integrating by parts and using the boundary conditions \eqref{eq:Neu}, we get
\begin{equation*}
	\frac{d}{dt}E_\e[u]=-\frac1\e\int_a^b\left[\left(\frac{\e^2 u_x}{\sqrt{1+\e^4u_x^2}}\right)_x-F'(u)\right]u_t\,dx.
\end{equation*}
Since $u$ satisfies the equation \eqref{eq:Q-model}, we end up with \eqref{eq:energy-dec}.
\end{proof}
In the following lemma we prove an inequality which plays in \eqref{eq:Q-model}-\eqref{eq:curv+} the same role that the well-known Young inequality plays in the metastability analysis of the classical reaction-diffusion model \eqref{eq:Al-Ca}.
\begin{lem}
Let $\e,L>0$. 
Then
\begin{equation}\label{eq:strangeineq+}
	\frac{\sqrt{1+\e^4x^2}-1}{\e^3}+\frac{y}{\e}\geq |x|\sqrt{2y-\e^2y^2},
\end{equation}
for any $(x,y)\in[-L,L]\times[0,2\e^{-2}]$.
\end{lem}
\begin{proof}
To prove \eqref{eq:strangeineq+} it is sufficient to study the sign of the function
\begin{equation*}
	f(x,y):=\sqrt{1+\e^4x^2}-1+\e^2y-\e^3x\sqrt{2y-\e^2y^2},
\end{equation*}
in the domain $D:=[0,L]\times[0,2\e^{-2}]$.
For any $(x,y)\in\mathring{D}$ we have
\begin{align*}
	& f_x(x,y)=\frac{\e^4x}{\sqrt{1+\e^4x^2}}-\e^3\sqrt{2y-\e^2y^2}, \\
	& f_y(x,y)=\e^2\left(1-\e x\frac{1-\e^2y}{\sqrt{2y-\e^2y^2}}\right).
\end{align*}
It follows that $\nabla f(x,y)=0$ if and only if 
\begin{equation*}
	x=\frac{\sqrt{2y-\e^2y^2}}{\e(1-\e^2y)}=:g(y), \qquad \qquad y\in(0,\e^{-2})\cup(\e^{-2},2\e^{-2}).
\end{equation*}
Notice that $f(g(y),y)=0$ for any $y\in(0,\e^{-2})\cup(\e^{-2},2\e^{-2})$, and so $f=0$ in all the critical points.
Let us compute the minimum of $f$ in $\partial D$.
First, we have
\begin{equation*}
	f(x,0)=\sqrt{1+\e^4x^2}-1, \qquad x\in[0,L], \qquad \qquad f(0,y)=\e^2y, \qquad y\in[0,2\e^{-2}].
\end{equation*}
Next, 
\begin{equation*}
	f(x,2\e^{-2})=\sqrt{1+\e^4x^2}+1, \qquad x\in[0,L].
\end{equation*}
Finally, $f'(L,y)\leq0$ for any $y\in(0,g^{-1}(L)]$ and $f'(L,y)\geq0$ for any $y\in[g^{-1}(L),2\e^{-2})$ and $f(L,g^{-1}(L))=0$.
Therefore, $f$ is non negative on $\partial D$ and since $f=0$ in all the critical points,
we have that $f$ is non negative in $D$ and the proof is complete.
\end{proof}
The inequality \eqref{eq:strangeineq+} is crucial because it allows us to state that if $\bar{u}$ is a monotone function
connecting the two stable points $+1$ and $-1$ and 
\begin{equation*}
	\max_{s\in[-1,1]} F(s)\leq2\e^{-2},
\end{equation*} then the energy defined in \eqref{eq:energy+} satisfies
\begin{equation}\label{eq:c_eps}
	E_\e[\bar{u}]\geq\int_a^b|\bar{u}'|\sqrt{F(\bar{u})\left[2-\e^2F(\bar{u})\right]}\,dx=\int_{-1}^{+1}\sqrt{F(s)\left[2-\e^2F(s)\right]}\,ds=:c_\e.
\end{equation}
As we will see, the positive constant $c_\e$ represents the minimum energy to have a transition between $-1$ and $+1$. It is to be observed that
\begin{equation*}
	\lim_{\e\to0} c_\e=\int_{-1}^{+1}\sqrt{2F(s)}\,ds=:c_0,
\end{equation*}
which is the minimum energy in the case of the classical Allen--Cahn equation (see \cite{BrKo90}).

We now give the definition of a function with a \emph{transition layer structure}.
\begin{defn}\label{def:TLS}
Let us fix $N\in\mathbb{N}$ and a {\it piecewise constant function} $v$ with $N$ transitions as follows:
\begin{equation}
\label{vstruct}
\begin{minipage}[c]{4in}
$v : [a,b] \rightarrow \{-1,+1\}$ with $N$ jumps located at $x = h_j$, $j = 1, \ldots, N$, with $a<h_1<h_2<\cdots<h_N<b$, and $r > 0$ is such that $(h_i-r,h_i+r)\cap(h_j-r,h_j+r)=\varnothing$ for $i \neq j$, $a\leq h_1-r$ and $h_N+r\leq b$.
\end{minipage}
\end{equation}
We say that a function $u^\e\in H^1(a,b)$ has an \emph{$N$-transition layer structure} if 
\begin{equation}\label{eq:ass-u0}
	\lim_{\varepsilon\rightarrow 0} \|u^\varepsilon-v\|_{{}_{L^1}}=0,
\end{equation}
and there exist $C>0$ and $A\in(0,r\sqrt{2\lambda})$, $\lambda=\min\{F''(\pm1)\}$ (independent on $\e$) such that
\begin{equation}\label{eq:energy-ini}
	E_\varepsilon[u^\varepsilon]\leq Nc_\e+C\exp(-A/\e),
\end{equation}
for any $\varepsilon\ll1$, where the energy $E_\e$ and the positive constant $c_ \e$ are defined in \eqref{eq:energy+} and in \eqref{eq:c_eps}, respectively.
\end{defn}
Observe that condition \eqref{eq:ass-u0} fixes the number of transitions and their relative positions in the limit $\varepsilon\to0$,
while condition \eqref{eq:energy-ini} requires that the energy at the time $t=0$ exceeds at most by $C\exp(-A/\e)$, the minimum possible to have these $N$ transitions. 
Moreover, from \eqref{eq:energy-dec} it follows that if the initial datum $u_0^\e$ satisfies \eqref{eq:energy-ini}, 
then the solution $u^\e(\cdot,t)$ satisfies the same inequality for all times $t>0$.
Therefore, we only need to prove that the solution $u^\e(\cdot,t)$ satisfies \eqref{eq:ass-u0} for an exponentially long time,
and this is the main result of this section.

\begin{thm}[metastable dynamics with strong saturating diffusion]\label{thm:main}
Assume that $F\in C^3(\R)$ satisfies \eqref{eq:ass-F} and define $\lambda:=\min\{F''(\pm1)\}$. 
Let $v$ be as in \eqref{vstruct} and let $A\in(0,r\sqrt{2\lambda})$. 
If $u^\varepsilon$ is the solution of \eqref{eq:Q-model}-\eqref{eq:Neu}-\eqref{eq:initial}
with $Q$ as in \eqref{eq:curv+} and initial datum $u_0^{\varepsilon}$ satisfying \eqref{eq:ass-u0} and \eqref{eq:energy-ini}, then, 
\begin{equation}\label{eq:limit}
	\sup_{0\leq t\leq m\exp(A/\varepsilon)}\|u^\varepsilon(\cdot,t)-v\|_{{}_{L^1}}\xrightarrow[\varepsilon\rightarrow0]{}0,
\end{equation}
for any $m>0$.
\end{thm}
As it was pointed out in Section \ref{sec:intro}, the crucial step in the proof of Theorem \ref{thm:main} is to prove a particular lower bound on the energy.
Such result is purely variational in character and equation \eqref{eq:Q-model} plays no role in its proof.
\begin{prop}\label{prop:lower}
Assume that $F\in C^3(\R)$ satisfies \eqref{eq:ass-F}. 
Let $v$ be as in \eqref{vstruct}, $\lambda:=\min\{F''(\pm1)\}>0$ and $A\in(0,r\sqrt{2\lambda})$.
Then, there exist $\e_0,C,\delta>0$ (depending only on $F,v$ and $A$) such that if $u\in H^1(a,b)$ satisfies 
\begin{equation}\label{eq:u-v}
	\|u-v\|_{{}_{L^1}}\leq\delta,
\end{equation}
then for any $\e\in(0,\e_0)$,
\begin{equation}\label{eq:lower}
	E_\varepsilon[u]\geq Nc_\e-C\exp(-A/\varepsilon),
\end{equation}
where $E_\e$ and $c_\e$ are defined in \eqref{eq:energy+} and \eqref{eq:c_eps}, respectively.
\end{prop}
\begin{proof}
Fix $u\in H^1(a,b)$ satisfying \eqref{eq:u-v}, and $\e>0$ such that $2-\e^2F(u)\geq0$ in $(a,b)$. 
Take $\hat r\in(0,r)$ and $\rho_1$ so small that 
\begin{equation}\label{eq:nu}
	A\leq(r-\hat r)\sqrt{2\lambda-2\nu\rho_1}, \qquad \mbox{ where } \quad \nu:=\sup\left\{|F'''(x)|, \,\, x\in[-1-\rho_1,1+\rho_1]\right\}.
\end{equation}
Then, choose $0<\rho_2 < \rho_1$  sufficiently small such that
\begin{equation}\label{eq:forrho2}
\begin{aligned}
	\int_{1-\rho_1}^{1-\rho_2}\sqrt{F(s)[2-\e^2F(s)]}\,ds&>\int_{1-\rho_2}^{1}\sqrt{F(s)[2-\e^2F(s)]}\,ds,  \\
	\int_{-1+\rho_2}^{-1+\rho_1}\sqrt{F(s)[2-\e^2F(s)]}\,ds&> \int_{-1}^{-1+\rho_2}\sqrt{F(s)[2-\e^2F(s)]}\,ds.
	\end{aligned}
\end{equation}
Now, we focus our attention on $h_i$, one of the discontinuous points of $v$. To fix ideas, 
let $v(h_i\pm r)=\pm1$, the other case being analogous.
We claim that there exist $r_+$ and $r_-$ in $(0,\hat r)$ such that
\begin{equation}\label{2points}
	|u(h_i+r_+)-1|<\rho_2, \qquad \quad \mbox{ and } \qquad \quad |u(h_i-r_-)+1|<\rho_2.
\end{equation}
Indeed, assume by contradiction that $|u-1|\geq\rho_2$ throughout $(h_i,h_i+\hat r)$; then
\begin{equation*}
	\delta\geq\|u-v\|_{{}_{L^1}}\geq\int_{h_i}^{h_i+\hat r}|u-v|\,dx\geq\hat r\rho_2,
\end{equation*}
and this leads to a contradiction if we choose $\delta\in(0,\hat r\rho_2)$.
Similarly, one can prove the existence of $r_-\in(0,\hat r)$ such that $|u(h_i-r_-)+1|<\rho_2$.

Now, we consider the interval $(h_i-r,h_i+r)$ and claim that
\begin{equation}\label{eq:claim}
	\int_{h_i-r}^{h_i+r}\left[\frac{\sqrt{1+\e^4u_x^2}-1}{\e^3}+\frac{F(u)}\e\right]\,dx\geq c_\e-\tfrac{C}N\exp(-A/\varepsilon),
\end{equation}
for some $C>0$ independent on $\e$.
Observe that from \eqref{eq:strangeineq+}, it follows that for any $a\leq c<d\leq b$,
\begin{equation}\label{eq:ineq}
	\int_c^d\left[\frac{\sqrt{1+\e^4u_x^2}-1}{\e^3}+\frac{F(u)}\e\right]\,dx \geq \left|\int_{u(c)}^{u(d)}\sqrt{F(s)[2-\e^2F(s)]}\,ds\right|.
\end{equation}
Hence, if $u(h_i+r_+)\geq1$ and $u(h_i-r_-)\leq-1$, then from \eqref{eq:ineq} we can conclude that
\begin{equation*}
	\int_{h_i-r_-}^{h_i+r_+}\left[\frac{\sqrt{1+\e^4u_x^2}-1}{\e^3}+\frac{F(u)}\e\right]\,dx\geq c_\e,
\end{equation*}
which implies \eqref{eq:claim}. On the other hand, from \eqref{eq:ineq} we obtain
\[
\int_{h_i-r_-}^{h_i+r_+}\left[\frac{\sqrt{1+\e^4u_x^2}-1}{\e^3}+\frac{F(u)}\e\right]\,dx \geq \int_{u(h_i - r_-)}^{u(h_i + r_+)} \sqrt{F(s)[2-\e^2F(s)]}\,ds,
\]
yielding, in turn,
%
\begin{align}
	\int_{h_i-r}^{h_i+r}\left[\frac{\sqrt{1+\e^4u_x^2}-1}{\e^3}+\frac{F(u)}\e\right]\,dx & 
	\geq \int_{h_i+r_+}^{h_i+r}\left[\frac{\sqrt{1+\e^4u_x^2}-1}{\e^3}+\frac{F(u)}\e\right]\,dx\notag \\ 
	& \quad + \int_{h_i-r}^{h_i-r_-}\left[\frac{\sqrt{1+\e^4u_x^2}-1}{\e^3}+\frac{F(u)}\e\right]\,dx \notag \\
	& \quad +\int_{-1}^{1}\sqrt{F(s)[2-\e^2F(s)]}\,ds\notag\\
	&\quad-\int_{-1}^{u(h_i-r_-)}\sqrt{F(s)[2-\e^2F(s)]}\,ds \notag \\
	& \quad-\int_{u(h_i+r_+)}^{1}\sqrt{F(s)[2-\e^2F(s)]}\,ds\notag \\
	&=:I_1+I_2+c_\e-\alpha_\e-\beta_\e. \label{eq:Pe}
\end{align}
Let us estimate the first two terms of \eqref{eq:Pe}. 
Regarding $I_1$, assume that $1-\rho_2<u(h_i+r_+)<1$ and consider the unique minimizer $z:[h_i+r_+,h_i+r]\rightarrow\R$ 
of $I_1$ subject to the boundary condition $z(h_i+r_+)=u(h_i+r_+)$.
If the range of $z$ is not contained in the interval $(1-\rho_1,1+\rho_1)$, then from \eqref{eq:ineq}, it follows that
\begin{equation}\label{E>fi}
	\int_{h_i+r_+}^{h_i+r}\left[\frac{\sqrt{1+\e^4z_x^2}-1}{\e^3}+\frac{F(z)}\e\right]\,dx \geq \int_{u(h_i+r_+)}^{1}\sqrt{F(s)[2-\e^2F(s)]}\,ds=\beta_\e,
\end{equation}
by the choice of $r_+$ and $\rho_2$. 
Suppose, on the other hand, that the range of $z$ is contained in the interval $(1-\rho_1,1+\rho_1)$. 
Then, the Euler-Lagrange equation for $z$ is
\begin{align*}
	&z''(x)=\e^{-2}F'(z(x))\left(1+\e^4z'(x)^2\right)^{3/2}, \quad \qquad x\in(h_i+r_+,h_i+r),\\
	&z(h_i+r_+)=u(h_i+r_+), \quad \qquad z'(h_i+r)=0.
\end{align*}
Denoting $\psi(x):=(z(x)-1)^2$, we have $\psi'=2(z-1)z'$ and 
\begin{equation*}
	\psi''(x)=2(z(x)-1)z''(x)+2z'(x)^2\geq\frac2{\varepsilon^2}(z(x)-1)F'(z(x))\left(1+\e^4z'(x)^2\right)^{3/2}.
\end{equation*}
Since $|z(x)-1|\leq\rho_1$ for any $x\in[h_i+r_+,h_i+r]$, using Taylor's expansion 
\begin{equation*}
	F'(z(x))=F''(1)(z(x)-1)+R,
\end{equation*}
where $|R|\leq\nu|z-1|^2/2$ with $\nu$ defined in \eqref{eq:nu}, we obtain
\[
\psi''(x) \geq \frac{2\lambda}{\varepsilon^2}(z(x)-1)^2-\frac{2\nu\rho_1}{\varepsilon^2}(z(x)-1)^2 \geq \frac{\mu^2}{\varepsilon^2}\psi(x),
\]
where $\mu=A/(r-\hat r)$ and having used \eqref{eq:nu}. 
Thus, $\psi$ satisfies
\begin{align*}
	\psi''(x)-\frac{\mu^2}{\varepsilon^2}\psi(x)\geq0, \quad \qquad x\in(h_i+r_+,h_i+r),\\
	\psi(h_i+r_+)=(u(h_i+r_+)-1)^2, \quad \qquad \psi'(h_i+r)=0.
\end{align*}
We compare $\psi$ with the solution $\hat \psi$ of
\begin{align*}
	\hat\psi''(x)-\frac{\mu^2}{\varepsilon^2}\hat\psi(x)=0, \quad \qquad x\in(h_i+r_+,h_i+r),\\
	\hat\psi(h_i+r_+)=(u(h_i+r_+)-1)^2, \quad \qquad \hat\psi'(h_i+r)=0,
\end{align*}
which can be explicitly calculated to be
\begin{equation*}
	\hat\psi(x)=\frac{(u(h_i+r_+)-1)^2}{\cosh\left[\frac\mu\varepsilon(r-r_+)\right]}\cosh\left[\frac\mu\varepsilon(x-(h_i+r))\right].
\end{equation*}
By the maximum principle, $\psi(x)\leq\hat\psi(x)$ so, in particular,
\begin{equation*}
	\psi(h_i+r)\leq\frac{(u(h_i+r_+)-1)^2}{\cosh\left[\frac\mu\varepsilon(r-r_+)\right]}\leq2\exp(-A/\varepsilon)(u(h_i+r_+)-1)^2.
\end{equation*}
Then, we have 
\begin{equation}\label{|z-v+|<exp}
	|z(h_i+r)-1|\leq\sqrt2\exp(-A/2\varepsilon)\rho_2.
\end{equation}
Now, by using Taylor's expansion for $F(s)$, we obtain
\begin{equation*}
	F(s)\leq(s-1)^2\left(\frac{F''(1)}2+\frac{o(|s-1|^2)}{|s-1|^2}\right).
\end{equation*}
Therefore, for $s$ sufficiently close to $s=1$ we have
\begin{equation}\label{W-quadratic}
	0\leq F(s)\leq\Lambda(s-1)^2,
\end{equation}
for some $\Lambda > 0$. Using \eqref{|z-v+|<exp} and \eqref{W-quadratic}, we obtain
\begin{equation}\label{fi<exp}
	\left|\int_{z(h_i+r)}^{1}\sqrt{F(s)[2-\e^2F(s)]}\,ds\right|\leq\sqrt{\Lambda/2}(z(h_i+r)-1)^2\leq\sqrt{2\Lambda}\,\rho_2^2\,\exp(-A/\varepsilon). 
\end{equation}
From \eqref{eq:ineq}-\eqref{fi<exp} it follows that, for some constant $C>0$, 
\begin{align}
	\int_{h_i+r_+}^{h_i+r}\left[\frac{\sqrt{1+\e^4z_x^2}-1}{\e^3}+\frac{F(z)}\e\right]\,dx &\geq \left|\int_{z(h_i+r_+)}^{1}\sqrt{F(s)[2-\e^2F(s)]}\,ds\,-\right.\nonumber \\
	&\qquad \qquad\left.\int_{z(h_i+r)}^{1}\sqrt{F(s)[2-\e^2F(s)]}\,ds\right| \nonumber\\
	& \geq\beta_\e-\tfrac{C}{2N}\exp(-A/\varepsilon). \label{E>fi-exp}
\end{align}
Combining \eqref{E>fi} and \eqref{E>fi-exp}, we get that the constrained minimizer $z$ of the proposed variational problem satisfies
\begin{equation*}	
	\int_{h_i+r_+}^{h_i+r}\left[\frac{\sqrt{1+\e^4z_x^2}-1}{\e^3}+\frac{F(z)}\e\right]\,dx \geq\beta_\e-\tfrac{C}{2N}\exp(-A/\varepsilon).
\end{equation*}
The restriction of $u$ to $[h_i+r_+,h_i+r]$ is an admissible function, so it must satisfy the same estimate and we have
\begin{equation}\label{eq:I1}
	I_1\geq\beta_\e-\tfrac{C}{2N}\exp(-A/\varepsilon).
\end{equation}
The term $I_2$ on the right hand side of \eqref{eq:Pe} is estimated similarly by analyzing  the interval $[h_i-r,h_i-r_-]$ 
and using the second condition of \eqref{eq:forrho2} to obtain the corresponding inequality \eqref{E>fi}.
The obtained lower bound reads
\begin{equation}\label{eq:I2}	
	I_2\geq\alpha_\e-\tfrac{C}{2N}\exp(-A/\varepsilon).
\end{equation}
Finally, by substituting \eqref{eq:I1} and \eqref{eq:I2} in \eqref{eq:Pe}, we deduce \eqref{eq:claim}.
Summing up all of these estimates for $i=1, \dots, N$, namely for all transition points, we end up with
\begin{equation*}
	E_\varepsilon[u]\geq\sum_{i=1}^N\int_{h_i-r}^{h_i+r}\left[\frac{\sqrt{1+\e^4u_x^2}-1}{\e^3}+\frac{F(u)}\e\right]\,dx\geq Nc_\e-C\exp(-A/\varepsilon),
\end{equation*}
and the proof is complete.
\end{proof}

Proposition \ref{prop:lower} permits to apply the energy approach introduced in \cite{BrKo90} and we can proceed as in \cite{BrKo90}, \cite{FLM19}.
\begin{prop}\label{prop:L2-norm}
Assume that $F\in C^3(\R)$ satisfies \eqref{eq:ass-F} and consider the solution $u^\e$ to \eqref{eq:Q-model}-\eqref{eq:Neu}-\eqref{eq:initial} 
with $Q$ given by \eqref{eq:curv+} and initial datum $u_0^{\varepsilon}$ satisfying \eqref{eq:ass-u0} and \eqref{eq:energy-ini}.
Then, there exist positive constants $\varepsilon_0, C_1, C_2>0$ (independent on $\varepsilon$) such that
\begin{equation}\label{L2-norm}
	\int_0^{C_1\varepsilon^{-1}\exp(A/\varepsilon)}\|u_t^\varepsilon\|^2_{{}_{L^2}}dt\leq C_2\varepsilon\exp(-A/\varepsilon),
\end{equation}
for all $\varepsilon\in(0,\varepsilon_0)$.
\end{prop}

\begin{proof}
Let $\varepsilon_0>0$ be sufficiently small such that for all $\varepsilon\in(0,\varepsilon_0)$, \eqref{eq:energy-ini} holds and 
\begin{equation}\label{1/2delta}
	\|u_0^\varepsilon-v\|_{{}_{L^1}}\leq\frac12\delta,
\end{equation}
where $\delta$ is the constant of Proposition \ref{prop:lower}. 
Let $\hat T>0$; we claim that if
\begin{equation}\label{claim1}
	\int_0^{\hat T}\|u_t^\varepsilon\|_{{}_{L^1}}dt\leq\frac12\delta,
\end{equation}
then there exists $C>0$ such that
\begin{equation}\label{claim2}
	E_\varepsilon[u^\varepsilon](\hat T)\geq Nc_\e-C\exp(-A/\varepsilon).
\end{equation}
Indeed, inequality \eqref{claim2} follows from Proposition \ref{prop:lower} if $\|u^\varepsilon(\cdot,\hat T)-v\|_{{}_{L^1}}\leq\delta$.
By using the triangle inequality, \eqref{1/2delta} and \eqref{claim1}, we obtain
\begin{equation*}
	\|u^\varepsilon(\cdot,\hat T)-v\|_{{}_{L^1}}\leq\|u^\varepsilon(\cdot,\hat T)-u_0^\varepsilon\|_{{}_{L^1}}+\|u_0^\varepsilon-v\|_{{}_{L^1}}
	\leq\int_0^{\hat T}\|u_t^\varepsilon\|_{{}_{L^1}}+\frac12\delta\leq\delta.
\end{equation*}
Upon integration of \eqref{eq:energy-dec}, we deduce
\begin{equation}\label{dissipative}
	E_\e[u^\e_0]-E_\e[u^\e](\hat T)=\e^{-1}\int_0^{\hat T}\|u_t^\e\|^2_{{}_{L^2}}\,dt.
\end{equation}
Substituting  \eqref{eq:energy-ini} and \eqref{claim2} in \eqref{dissipative} yields
\begin{equation}\label{L2-norm-Teps}
	\int_0^{\hat T}\|u_t^\varepsilon\|^2_{{}_{L^2}}dt\leq C_2\e\exp(-A/\varepsilon).
\end{equation}
It remains to prove that inequality \eqref{claim1} holds for $\hat T\geq C_1\e^{-1}\exp(A/\varepsilon)$.
If 
\begin{equation*}
	\int_0^{+\infty}\|u_t^\varepsilon\|_{{}_{L^1}}dt\leq\frac12\delta,
\end{equation*}
there is nothing to prove. 
Otherwise, choose $\hat T$ such that
\begin{equation*}
	\int_0^{\hat T}\|u_t^\varepsilon\|_{{}_{L^1}}dt=\frac12\delta.
\end{equation*}
Using H\"older's inequality and \eqref{L2-norm-Teps}, we infer
\begin{equation*}
	\frac12\delta\leq[\hat T(b-a)]^{1/2}\biggl(\int_0^{\hat T}\|u_t^\varepsilon\|^2_{{}_{L^2}}dt\biggr)^{1/2}\leq
	\left[\hat T(b-a)C_2\varepsilon\exp(-A/\varepsilon)\right]^{1/2}.
\end{equation*}
It follows that there exists $C_1>0$ such that
\begin{equation*}
	\hat T\geq C_1\varepsilon^{-1}\exp(A/\varepsilon),
\end{equation*}
and the proof is complete.
\end{proof}

Now, we have all the tools to prove \eqref{eq:limit}.
\begin{proof}[Proof of Theorem \ref{thm:main}]
The triangle inequality yields
\begin{equation}\label{trianglebar}
	\|u^\varepsilon(\cdot,t)-v\|_{{}_{L^1}}\leq\|u^\varepsilon(\cdot,t)-u_0^\varepsilon\|_{{}_{L^1}}+\|u_0^\varepsilon-v\|_{{}_{L^1}},
\end{equation}
for all $t\in[0,m\exp(A/\varepsilon)]$. 
The last term of inequality \eqref{trianglebar} tends to zero by assumption \eqref{eq:ass-u0}.
Regarding the first term, take $\varepsilon$ so small that $C_1\varepsilon^{-1}\geq m$;
thus we can apply Proposition \ref{prop:L2-norm} and by using H\"older's inequality and \eqref{L2-norm}, we infer
\begin{equation*}
	\sup_{0\leq t\leq m\exp(A/\varepsilon)}\|u^\e(\cdot,t)-u^\e_0\|_{{}_{L^1}}\leq\int_0^{m\exp(A/\varepsilon)}\|u_t^\e(\cdot,t)\|_{{}_{L^1}}\,dt\leq C\sqrt\e,
\end{equation*}		
for all $t\in[0,m\exp(A/\varepsilon)]$. Hence \eqref{eq:limit} follows.
\end{proof}

Theorem \ref{thm:main} provides sufficient conditions for the existence of a metastable state for equation \eqref{eq:Q-model}-\eqref{eq:curv+} 
and shows its persistence for (at least) an exponentially long time.
We conclude this section by constructing a family of functions with a transition layer structure, 
namely we show how to construct a function satisfying \eqref{eq:ass-u0}-\eqref{eq:energy-ini}.
To do this, we will use standing waves solutions to \eqref{eq:Q-model2}, i.e. we will use the solution $\Phi_\e=\Phi_\e(x)$ to the boundary value problem
\begin{equation}\label{eq:Fi}
	\e^2Q'(\e^2\Phi'_\e)\Phi''_\e-F'(\Phi_\e)=0, \qquad \qquad \lim_{x\to\pm\infty}\Phi_\e(x)=\pm1, \qquad \Phi_\e(0)=0,
\end{equation}
where $Q$ is defined in \eqref{eq:curv+}.
\begin{prop}\label{prop:ex-met}
Fix a piecewise function $v$ as in \eqref{vstruct} and assume that $F\in C^3(\R)$ satisfies \eqref{eq:ass-F}.
Then there exists a function $u^\e$ satisfying \eqref{eq:ass-u0} and \eqref{eq:energy-ini}.
\end{prop}
\begin{proof}
First of all, let us prove that if $Q$ is as in \eqref{eq:curv+} and $F\in C^3(\R)$ satisfies \eqref{eq:ass-F} and 
\begin{equation}\label{eq:maxF}
	\max_{\Phi\in[-1,1]}F(\Phi)<\e^{-2},
\end{equation}
then, there exists a unique solution to \eqref{eq:Fi}.
Multiplying by $\Phi'_\e=\Phi'_\e(x)$ equation \eqref{eq:Fi}, we deduce
\begin{equation*}
	\e^2Q'(\e^2\Phi'_\e)\Phi'_\e\Phi''_\e-F'(\Phi_\e)\Phi'_\e=0, \qquad\qquad \mbox{in }\, (-\infty,+\infty),
\end{equation*}
and, therefore,
\begin{equation*}
	P_\e(\Phi'_\e)=F(\Phi_\e), \qquad\qquad \mbox{in }\, (-\infty,+\infty),
\end{equation*}
where
\begin{equation*}
	P_\e(s):=\int_{0}^{s}\e^2u\,Q'(\e^2u)\,du=\frac{1}{\e^2}\left(1-\frac{1}{\sqrt{1+\e^4s^2}}\right).
\end{equation*}
It follows that the profile $\Phi_\e$ satisfies 
\begin{equation}\label{eq:FI-first}
	\begin{cases}
		\e\Phi'_\e(x)=\displaystyle\frac{\sqrt{F(\Phi_\e)[2-\e^2F(\Phi_\e)]}}{1-\e^2F(\Phi_\e)},\\
		\Phi_\e(0)=0.
	\end{cases}
\end{equation}
Hence, the assumptions on the potential $F$ \eqref{eq:ass-F} imply that there exists a unique solution to \eqref{eq:FI-first} which is increasing and implicitly defined by
\begin{equation*}
	\int_{0}^{\Phi_\e(x)}\displaystyle\frac{1-\e^2F(s)}{\sqrt{F(s)[2-\e^2F(s)]}}\,ds=\frac{x}{\e}.
\end{equation*} 
Observe that 
\begin{equation*}
	\lim_{\e\to0}\Phi_\e(x)=
		\begin{cases}
			-1, \qquad & x<0,\\	
			0, &x=0,\\
			+1, & x>0.
		\end{cases}
\end{equation*}
Now, choose $\e_0>0$ small enough so that the condition \eqref{eq:maxF} holds for any $\e\in(0,\e_0)$.
We use the profile $\Phi_\e$ to construct a family of functions with a transition layer structure.
Fix $N\in\mathbb{N}$ and $N$ transition points $a<h_1<h_2<\dots<h_n<b$, and denote the middle points by
\begin{equation*}
	m_1:=a, \qquad \quad m_j:=\frac{h_{j-1}+h_j}{2}, \quad j=2,\dots,N-1, \qquad \quad m_N:=b.
\end{equation*}
Define
\begin{equation}\label{eq:translayer}
	u^\e(x):=\Phi_\e\left((-1)^j(x-h_j)\right), \qquad \qquad x\in[m_j,m_{j+1}], \qquad \qquad j=1,\dots N,
\end{equation}
where $\Phi_\e$ is the solution to \eqref{eq:Fi}.
Notice that $u^\e(h_j)=0$, for $j=1,\dots,N$ and for definiteness we choose $u^\e(a)<0$ (the case $u^\e(a)>0$ is analogous).
Let us prove that $u^\e$ has an $N$-transition layer structure, i.e. that it satisfies \eqref{eq:ass-u0}-\eqref{eq:energy-ini}.
It is easy to check that $u^\e\in H^1(a,b)$ and satisfies \eqref{eq:ass-u0}; let us show that
\begin{equation}\label{eq:Nc_eps}
	E_\e[u^\e]\leq Nc_\e,
\end{equation}
where $E_\e$ and $c_\e$ are defined in \eqref{eq:energy+} and \eqref{eq:c_eps}, respectively.
From the definitions of $E_\e$ and $u^\e$ we obtain
\begin{equation*}
	E_\e[u^\e]=\sum_{j=1}^{N}\int_{m_j}^{m_{j+1}}\left[\frac{\sqrt{1+\e^4\left(\Phi'_\e\right)^2}-1}{\e^3}+\frac{F(\Phi_\e)}\e\right]\,dx.
\end{equation*}
From \eqref{eq:FI-first}, it follows that
\begin{align*}
	\int_{m_j}^{m_{j+1}}\!\!\left[\frac{\sqrt{1+\e^4\left(\Phi'_\e\right)^2}-1}{\e^3}+\frac{F(\Phi_\e)}\e\right]dx&=
	\int_{m_j}^{m_{j+1}}\!\!\left[ \frac1{\e^3-\e^5F(\Phi_\e)}-\frac{1}{\e^3}+\frac{F(\Phi_\e)}\e\right]dx\\
	&=\int_{m_j}^{m_{j+1}} \frac{F(\Phi_\e)[2-\e^2F(\Phi_\e)]}{\e[1-\e^2F(\Phi_\e)]}\,dx\\
	&=\int_{m_j}^{m_{j+1}}\Phi'_\e\sqrt{F(\Phi_\e)[2-\e^2F(\Phi_\e)]}\,dx\\
	&=\int_{\Phi_\e(m_j)}^{\Phi_\e(m_{j+1})}\sqrt{F(s)[2-\e^2F(s)]}\,ds<c_\e,
\end{align*}
where $c_\e$ is defined in \eqref{eq:c_eps}.
Summing up all the terms we end up with \eqref{eq:Nc_eps} and the proof is complete.
\end{proof}
\begin{rem}\label{rem:standing}
Let us stress out that the assumption $F''(\pm1)>0$ implies the exponential decay
\begin{equation}\label{eq:exp-decay}
	\begin{aligned}
		&1-\Phi_\e(x)\leq c_1e^{-c_2x/\e},  \qquad &\mbox{as } x\to+\infty,\\ 
		&\Phi_\e(x)+1\leq c_1e^{c_2x/\e}, &\mbox{as } x\to-\infty,
	\end{aligned}
\end{equation}
for some constants $c_1,c_2>0$ (depending on $F$).
On the other hand, if $F$ is a double well potential with wells at $\pm1$ and $F''(\pm1)=0$, then we have existence of a unique solution to \eqref{eq:Fi} 
but we do not have the exponential decay \eqref{eq:exp-decay}.
As we will see in Section \ref{sec:num}, there are considerable differences on the metastable dynamics of the solutions to \eqref{eq:Q-model} 
between the case $F''(\pm1)>0$ and the \emph{degenerate} case $F''(\pm1)=0$;
indeed, the assumption $F''(\pm1)>0$ is necessary to have persistence of metastable states for an exponentially long time.
\end{rem}

\section{Slow motion in the case of an unbounded diffusion}\label{sec:met-}
In this section we apply the same strategy of the previous one in order to prove existence 
and persistence for an exponentially long time of metastable states for the IBVP \eqref{eq:Q-model}-\eqref{eq:Neu}-\eqref{eq:initial}, 
when $Q$ is the mean curvature operator in Lorentz--Minkowski space \eqref{eq:curv-}. 
Most of the results of Section \ref{sec:met+} hold also when considering $Q$ as in \eqref{eq:curv-} and the energy \eqref{eq:energy-}.
For instance, by reasoning as in the proof of \eqref{eq:energy-dec}, one can prove that if $u\in C([0,T],H^2(a,b))$ is a solution of \eqref{eq:Q-model}-\eqref{eq:Neu}, 
with $Q$ as in \eqref{eq:curv-}, and $\mathcal{E}_\e$ is the energy defined in \eqref{eq:energy-}, then
\begin{equation}\label{eq:energy-var-}
	\mathcal{E}_\e[u](0)-\mathcal{E}_\e[u](T)=\e^{-1}\int_0^T\|u_t(\cdot,t)\|^2_{{}_{L^2}}dt.
\end{equation}
In particular, by using the inequality
\begin{equation*}
	\frac12\e x^2\leq\frac{1-\sqrt{1-\e^4x^2}}{\e^3}, \qquad \qquad x\in(-\e^{-2},\e^{-2}),
\end{equation*}
and the positiveness of $F$, we obtain the following \emph{a-priori} estimate 
\begin{equation*}
	\|u_x(\cdot,t)\|^2_{{}_{L^2}}\leq C\e^{-1}, \qquad\qquad \forall\, t\in[0,T],
\end{equation*}
for some positive constant $C$ independent on $\e$.

In order to prove the exponentially slow motion of the solutions we make use of the following inequality (which plays the same role of \eqref{eq:strangeineq+} in the case of saturating diffusion).
\begin{lem}
	Set $\e,L>0$. 
	Then
	\begin{equation}\label{eq:strangeineq-}
	\frac{1-\sqrt{1-\e^4x^2}}{\e^3}+\frac{y}{\e}\geq |x|\sqrt{2y+\e^2y^2},
	\end{equation}
	for any $(x,y)\in[-\e^{-2},\e^{-2}]\times[0,L]$.
\end{lem}
\begin{proof}
Similarly as the proof of \eqref{eq:strangeineq+}, we study the sign of the function
\begin{equation*}
	f(x,y):=1-\sqrt{1-\e^4x^2}+\e^2y-\e^3x\sqrt{2y+\e^2y^2},
\end{equation*}
in the domain $K:=[0,\e^{-2}]\times[0,L]$.
For any $(x,y)\in\mathring{K}$, we have
\begin{align*}
	& f_x(x,y)=\frac{\e^4x}{\sqrt{1-\e^4x^2}}-\e^3\sqrt{2y+\e^2y^2}, \\
	& f_y(x,y)=\e^2\left(1-\e x\frac{1+\e^2y}{\sqrt{2y+\e^2y^2}}\right),
\end{align*}
and $\nabla f(x,y)=0$ if and only if 
\begin{equation*}
	x=\frac{\sqrt{2y+\e^2y^2}}{\e(1+\e^2y)}=:\bar g(y), \qquad \qquad y\in(0,L).
\end{equation*}
Notice that $f(\bar g(y),y)=0$ for any $y\in[0,L]$, and so $f=0$ in all the critical points.
Regarding the boundary of the domain $K$, we have
\begin{equation*}
	f(0,y)\geq0, \quad \forall\,y\in[0,L], \qquad \qquad 
	f(x,0)\geq0, \quad \forall\,x\in[0,\e^{-2}].
\end{equation*}
Moreover, one has
\begin{equation*}
	f(\e^{-2},y)=1+\e^2y-\e\sqrt{2y+\e^2y^2}>0, \qquad \qquad \forall\,y\in[0,L],
\end{equation*}
Finally, it is easy to check that $f'(x,L)\leq0$ in $[0,\bar g(L)]$, $f'(x,L)\geq0$ in $[\bar g(L),\e^{-2})$ and $f(\bar g(L),L)=0$.
Therefore $f$ is non negative in $K$ and the inequality \eqref{eq:strangeineq-} follows.
\end{proof}
Similarly as \eqref{eq:c_eps}, the inequality \eqref{eq:strangeineq-} allows us to state that if $\bar{u}$ is a monotone function
connecting the two stable points $+1$ and $-1$, then for the energy defined in \eqref{eq:energy-} one has
\begin{equation}\label{eq:gamma_eps}
	\mathcal{E}_\e[\bar{u}]\geq\int_a^b|\bar{u}'|\sqrt{F(\bar{u})\left[2+\e^2F(\bar{u})\right]}\,dx=\int_{-1}^{+1}\sqrt{F(s)\left[2+\e^2F(s)\right]}\,ds=:\gamma_\e.
\end{equation}
Therefore, in this case the minimum energy to have a transition is given by $\gamma_\e$ and as before
\begin{equation*}
	\lim_{\e\to0} \gamma_\e=\int_{-1}^{+1}\sqrt{2F(s)}\,ds=:c_0.
\end{equation*}
As in Section \ref{sec:met+}, we say that a function $u^\e$ has an $N$-transition layer structure if \eqref{eq:ass-u0} holds with $v$ defined in \eqref{eq:translayer} and
and there exist $C>0$ and $A\in(0,r\sqrt{2\lambda})$, $\lambda=\min\{F''(\pm1)\}$ (independent on $\e$) such that
\begin{equation}\label{eq:energy-ini-}
	\mathcal{E}_\varepsilon[u^\varepsilon]\leq N\gamma_\e+C\exp(-A/\e),
\end{equation}
for any $\varepsilon\ll1$, where the energy $\mathcal{E}_\e$ and the positive constant $\gamma_ \e$ are defined in \eqref{eq:energy-} and \eqref{eq:gamma_eps}, respectively.

An example of function satisfying \eqref{eq:ass-u0}-\eqref{eq:energy-ini-} is given by \eqref{eq:translayer}, 
where we substitute $\Phi_\e$ with the solution $\Psi_\e:=\Psi_\e(x)$ of the boundary value problem
\begin{equation}\label{eq:Psi}
	\e^2Q'(\e^2\Psi'_\e)\Psi''_\e-F'(\Psi_\e)=0, \qquad \qquad \lim_{x\to\pm\infty}\Psi_\e(x)=\pm1, \qquad \Psi_\e(0)=0,
\end{equation}
where now $Q$ is defined in \eqref{eq:curv-}.
\begin{lem}
	Let $Q$ be as in \eqref{eq:curv-} and assume that $F\in C^3(\R)$ satisfies \eqref{eq:ass-F}. 
	Then, there exists a unique solution to \eqref{eq:Psi}.
\end{lem}
\begin{proof}
By reasoning as in the proof of Proposition \ref{prop:ex-met} and multiplying by $\Psi'_\e=\Psi'_\e(x)$ equation \eqref{eq:Psi}, we deduce
\begin{equation*}
\frac{1}{\sqrt{1-\e^4(\Psi'_\e)^2}}=1+\e^2F(\Psi_\e), \qquad\qquad \mbox{in }\, (-\infty,+\infty).
\end{equation*}
Consequently, the profile $\Phi_\e$ satisfies 
\begin{equation*}
	\begin{cases}
	\e\Psi'_\e(x)=\displaystyle\frac{\sqrt{F(\Psi_\e)[2+\e^2F(\Psi_\e)]}}{1+\e^2F(\Psi_\e)},\\
	\Psi_\e(0)=0.
	\end{cases}
\end{equation*}
Since the potential $F$ satisfies \eqref{eq:ass-F}, we have a unique solution to the boundary value problem \eqref{eq:Psi}, 
which is strictly increasing and implicitly defined by
\begin{equation*}
	\int_{0}^{\Psi_\e(x)}\displaystyle\frac{1+\e^2F(s)}{\sqrt{F(s)[2+\e^2F(s)]}}\,ds=\frac{x}{\e},
\end{equation*} 
and the proof is complete.
\end{proof}
As in Remark \ref{rem:standing}, we emphasize the fact that the assumption $F''(\pm1)>0$ implies an exponential decay like \eqref{eq:exp-decay}.
Moreover, we have
\begin{equation*}
	\lim_{\e\to0}\Psi_\e(x)=
		\begin{cases}
			-1, \qquad & x<0,\\	
			0, &x=0,\\
			+1, & x>0,
		\end{cases}
\end{equation*}
and for any $\e_0>0$ there exists $C>0$ (independent on $\e$) such that
\begin{equation*}
	\Psi'_\e(x)\leq C\e^{-1}, \qquad \qquad\forall\, x\in\R,
\end{equation*}
for any $\e\in(0,\e_0)$.

In particular, one can observe (similarly to \eqref{eq:Nc_eps}) that  the function $U^\e$, defined by
\begin{equation*}
	U^\e(x):=\Psi_\e\left((-1)^j(x-h_j)\right), \qquad \qquad x\in[m_j,m_{j+1}], \qquad \qquad j=1,\dots N,
\end{equation*}
satisfies
\begin{equation}\label{eq:Ngamma_eps}
	\mathcal{E}_\e[U^\e]\leq N\gamma_\e.
\end{equation}

Now, we can proceed in the same way as in the analysis of Section \ref{sec:met+} to prove the persistence for an exponentially long time of the $N$-transition layer structure.
\begin{thm}[metastable dynamics with unbounded diffusion]\label{thm:main-}
Assume that $F\in C^3(\R)$ satisfies \eqref{eq:ass-F} and define $\lambda:=\min\{F''(\pm1)\}$. 
Let $v$ be as in \eqref{vstruct} and let $A\in(0,r\sqrt{2\lambda})$. 
If $u^\varepsilon$ is the solution of \eqref{eq:Q-model}-\eqref{eq:Neu}-\eqref{eq:initial}
with $Q$ as in \eqref{eq:curv-} and initial datum $u_0^{\varepsilon}$ satisfying \eqref{eq:ass-u0} and \eqref{eq:energy-ini-}, then, 
\begin{equation}\label{eq:main}
	\sup_{0\leq t\leq m\exp(A/\varepsilon)}\|u^\varepsilon(\cdot,t)-v\|_{{}_{L^1}}\xrightarrow[\varepsilon\rightarrow0]{}0,
\end{equation}
for any $m>0$.
\end{thm}
The proof of Theorem \ref{thm:main-} depends upon the following lower bound on the energy functional $\mathcal{E}_\e$. 
\begin{prop}\label{prop:lower-}
Assume that $F\in C^3(\R)$ satisfies \eqref{eq:ass-F}. 
Let $v$ be as in \eqref{vstruct}, $\lambda:=\min\{F''(\pm1)\}>0$ and $A\in(0,r\sqrt{2\lambda})$.
Then, there exist $\e_0,C,\delta>0$ (depending only on $F,v$ and $A$) such that if $u\in H^1(a,b)$ satisfies 
\begin{equation*}
	\|u-v\|_{{}_{L^1}}\leq\delta,
\end{equation*}
then for any $\e\in(0,\e_0)$,
\begin{equation*}
	\mathcal{E}_\varepsilon[u]\geq N\gamma_\e-C\exp(-A/\varepsilon),
\end{equation*}
where $\mathcal{E}_\e$ and $\gamma_\e$ are defined in \eqref{eq:energy-} and \eqref{eq:gamma_eps}, respectively.
\end{prop}
Proposition \ref{prop:lower-} can be proved with the same technique used to prove Proposition \ref{prop:lower};
indeed, the computations are very similar, with the appropriate changes due to the fact than here the energy functional is \eqref{eq:energy-}.

Notice that from Proposition \ref{prop:lower-} and \eqref{eq:Ngamma_eps}, it follows that
\begin{equation*}
	N\gamma_\e-C\exp(-A/\varepsilon)\leq\mathcal{E}_\e[U^\e]\leq N\gamma_\e,
\end{equation*}
and, as a trivial consequence
\begin{equation*}
	\lim_{\e\to0}\mathcal{E}_\e[U^\e]=N c_0.
\end{equation*}
Thanks to Proposition \ref{prop:lower-} and the energy dissipation \eqref{eq:energy-var-}, we can proceed
as in Section \ref{sec:met+} to prove Theorem \ref{thm:main-}.
Indeed, we can prove the same estimate \eqref{L2-norm} also for the solutions to \eqref{eq:Q-model} with $Q$ given by \eqref{eq:curv-}
and conclude \eqref{eq:main} (the proof is formally identical, and it is sufficient to substitute the energy $E_\e$ with $\mathcal{E}_\e$; we omit the details).

\subsection{Layer Dynamics}
In this section, we give an estimate on the velocity of the transition points $h_1,\ldots,h_N$.
More precisely, we will show that such velocity is at most exponentially small as $\e\to0$.
We shall consider the case when $Q$ is given by \eqref{eq:curv-} and so the energy $\mathcal{E}_\e$, but the results can be easily extended to the case \eqref{eq:curv+}.

Fix $v$ as in \eqref{vstruct} and define its {\it interface} $I[v]$ as
\begin{equation*}
	I[v]:=\{h_1,h_2,\ldots,h_N\}.
\end{equation*}
For an arbitrary function $u:[a,b]\rightarrow\mathbb{R}$ and an arbitrary closed subset $K\subset\R\backslash\{\pm1\}$,
the {\it interface} $I_K[u]$ is defined by
\begin{equation*}
	I_K[u]:=u^{-1}(K).
\end{equation*}
Finally, we recall that for any $A,B\subset\mathbb{R}$ the {\it Hausdorff distance} $d(A,B)$ between $A$ and $B$ is defined by 
\begin{equation*}
	d(A,B):=\max\biggl\{\sup_{\alpha\in A}d(\alpha,B),\,\sup_{\beta\in B}d(\beta,A)\biggr\},
\end{equation*}
where $d(\beta,A):=\inf\{|\beta-\alpha|: \alpha\in A\}$. 

The following result is purely variational in character and states that, if a function $u\in H^1([a,b])$ is close to $v$ in $L^1$ and 
$\mathcal{E}_\varepsilon[u]$ exceeds of a small quantity the minimum energy to have $N$ transitions, then 
the distance between the interfaces $I_K[u]$ and $I_K[v]$ is small.  
\begin{lem}\label{lem:interface}
Assume that $F\in C^3(\R)$ satisfies \eqref{eq:ass-F} and let $v$ be as in \eqref{vstruct}.
Given $\delta_1\in(0,r)$ and a closed subset $K\subset\R\backslash\{\pm1\}$, 
there exist positive constants $\hat\delta,\varepsilon_0$ (independent on $\e$) and $M_\e>0$ such that for any $u\in H^1([a,b])$ satisfying
\begin{equation}\label{eq:lem-interf}
	\|u-v\|_{{}_{L^1}}<\hat\delta \qquad \quad \mbox{ and } \qquad \quad \mathcal{E}_\varepsilon[u]\leq N\gamma_\e+M_\e,
\end{equation}
for all $\varepsilon\in(0,\varepsilon_0)$, we have
\begin{equation}\label{lem:d-interfaces}
	d(I_K[u], I[v])<\tfrac12\delta_1.
\end{equation}
\end{lem}
\begin{proof}
Fix $\delta_1\in(0,r)$ and choose $\rho>0$ small enough that 
\begin{equation*}
	I_\rho:=(-1-\rho,-1+\rho)\cup(1-\rho,1+\rho)\subset\R\backslash K, 
\end{equation*}
and 
\begin{equation*}
	\inf\left\{\left|\int_{\xi_1}^{\xi_2}\sqrt{F(s)[2+\e^2F(s)]}\,ds\right| : \xi_1\in K, \xi_2\in I_\rho\right\}>2M_\e,
\end{equation*}
where
\begin{equation*}
	M_\e:=2N\max\left\{\int_{1-\rho}^{1}\sqrt{F(s)[2+\e^2F(s)]}\,ds, \, \int_{-1}^{-1+\rho}\sqrt{F(s)[2+\e^2F(s)]}\,ds \right\}.
\end{equation*}
By reasoning as in the proof of \eqref{2points} in Proposition \ref{prop:lower}, we can prove that for each $i$ there exist
\begin{equation*}
	x^-_{i}\in(h_i-\delta_1/2,h_i) \qquad \textrm{and} \qquad x^+_{i}\in(h_i,h_i+\delta_1/2),
\end{equation*}
such that
\begin{equation*}
	|u(x^-_{i})-v(x^-_{i})|<\rho \qquad \textrm{and} \qquad |u(x^+_{i})-v(x^+_{i})|<\rho.
\end{equation*}
Suppose that \eqref{lem:d-interfaces} is violated. 
Using \eqref{eq:strangeineq-}, we deduce
\begin{align}
	\mathcal{E}_\varepsilon[u]\geq&\sum_{i=1}^N\left|\int_{u(x^-_{i})}^{u(x^+_{i})}\sqrt{F(s)[2+\e^2F(s)]}\,ds\right|\notag\\ 
	& \qquad +\inf\left\{\left|\int_{\xi_1}^{\xi_2}\sqrt{2F(s)[2+\e^2F(s)]}\,ds\right| : \xi_1\in K, \xi_2\in I_\rho\right\}. \label{diseq:E1}
\end{align}
On the other hand, we have
\begin{align*}
	\left|\int_{u(x^-_{i})}^{u(x^+_{i})}\sqrt{2F(s)[2+\e^2F(s)]}\,ds\right|&\geq\int_{-1}^{1}\sqrt{2F(s)[2+\e^2F(s)]}\,ds\\
	&\qquad-\int_{-1}^{-1+\rho}\sqrt{2F(s)[2+\e^2F(s)]}\,ds\\
	&\qquad -\int_{1-\rho}^{1}\sqrt{2F(s)[2+\e^2F(s)]}\,ds\\
	&\geq\gamma_\e-\frac{M_\e}{N}. 
\end{align*}
Substituting the latter bound in \eqref{diseq:E1}, we deduce
\begin{equation*}
	\mathcal{E}_\varepsilon[u]\geq N\gamma_\e-M_\e+\inf\left\{\left|\int_{\xi_1}^{\xi_2}\sqrt{2F(s)[2+\e^2F(s)]}\,ds\right| : \xi_1\in K, \xi_2\in I_\rho\right\}.
\end{equation*}
For the choice of $\rho$,  we obtain	
\begin{align*}
	\mathcal{E}_\varepsilon[u]>N\gamma_\e+M_\e,
\end{align*}
which is a contradiction with assumption \eqref{eq:lem-interf}. Hence, the bound \eqref{lem:d-interfaces} is true.
\end{proof}

Thank to Theorem \ref{thm:main-} and Lemma \ref{lem:interface} we can prove the following result, which gives an upper bound on the velocity of the transition points.
\begin{thm}\label{thm:interface}
Assume that $F\in C^3(\R)$ satisfies \eqref{eq:ass-F}.
Let $u^\varepsilon$ be the solution of \eqref{eq:Q-model}-\eqref{eq:initial}-\eqref{eq:Neu}, with $Q$ given by \eqref{eq:curv-} and 
initial datum $u_0^{\varepsilon}$ satisfying \eqref{eq:ass-u0} and \eqref{eq:energy-ini-}. 
Given $\delta_1\in(0,r)$ and a closed subset $K\subset\R\backslash\{\pm1\}$, set
\begin{equation*}
	t_\varepsilon(\delta_1)=\inf\{t:\; d(I_K[u^\varepsilon(\cdot,t)],I_K[u_0^\varepsilon])>\delta_1\}.
\end{equation*}
There exists $\varepsilon_0>0$ such that if $\varepsilon\in(0,\varepsilon_0)$ then
\begin{equation*}
	t_\varepsilon(\delta_1)>\exp(A/\varepsilon).
\end{equation*}
\end{thm}
\begin{proof}
Let $\varepsilon_0>0$ be so small that \eqref{eq:ass-u0} and \eqref{eq:energy-ini-}
imply $u_0^\varepsilon$ satisfies \eqref{eq:lem-interf} for all $\varepsilon\in(0,\varepsilon_0)$.
From Lemma \ref{lem:interface} it follows that
\begin{equation}\label{interfaces-u0}
	d(I_K[u_0^\varepsilon], I[v])<\tfrac12\delta_1.
\end{equation}
Now, consider $u^\varepsilon(\cdot,t)$ for all $t\leq\exp(A/\varepsilon)$.
Assumption \eqref{eq:lem-interf} is satisfied thanks to \eqref{eq:main} and because $E_\varepsilon[u^\varepsilon](t)$ is a non-increasing function of $t$.
Then,
\begin{equation}\label{interfaces-u}
	d(I_K[u^\varepsilon(t)], I[v])<\tfrac12\delta_1
\end{equation}
for all $t\in(0,\exp(A/\varepsilon))$. 
Combining \eqref{interfaces-u0} and \eqref{interfaces-u}, we obtain
\begin{equation*}
	d(I_K[u^\varepsilon(t)],I_K[u_0^\varepsilon])<\delta_1,
\end{equation*}
for all $t\in(0,\exp(A/\varepsilon))$.
\end{proof}

\section{Numerical solutions}\label{sec:num}
In this section we present the results of some numerical calculations for the solutions to \eqref{eq:Q-model}-\eqref{eq:Neu} 
in the cases where the diffusion function $Q$ is determined by both \eqref{eq:curv+} and \eqref{eq:curv-}. 
The computed solutions confirm the analytical results of Sections \ref{sec:met+} and \ref{sec:met-}.
Moreover, we numerically show that the metastable states are attractive for a wide class of initial data,
that is, we show examples where the initial datum $u_0$ does not have a transition layer structure, 
but, after a very short time $T>0$, the solution reaches  such a structure and exhibits a metastable dynamics.

Finally, we consider the case when the assumptions \eqref{eq:ass-F} on $F$  are not satisfied and we
show that they are crucial for the appearance of the phenomenon of metastability.
More precisely, we consider the \emph{degenerate case}, $F''(\pm1)=0$, and we show that 
the solutions maintain a transition layer structure for a time much smaller than in the case with $F''(\pm1)>0$.

\subsection{Numerical experiment no. 1}

Let us start with an example where all the assumptions of Sections \ref{sec:met+} and \ref{sec:met-} are satisfied. 
The first numerical simulation considers a potential of the form 
\begin{equation}
\label{dwpot}
F(u)=\frac14(u^2-1)^2,
\end{equation}
which is the simplest example of function satisfying \eqref{eq:ass-F}, as well as an initial datum $u_0$ having a transition layer structure. Figure \ref{fig:6trans} depicts the plots of the computed numerical solutions to \eqref{eq:Q-model}-\eqref{eq:Neu}-\eqref{eq:initial} with $\e=0.1$ and for such potential $F$. The selected initial datum underlies a $6$-transition layer structure with specific transition points.

\begin{figure}[hbtp]
\centering
\subfigure[Saturating diffusion: $Q$ given by \eqref{eq:curv+}]{\label{1left}\includegraphics[scale = 0.38]{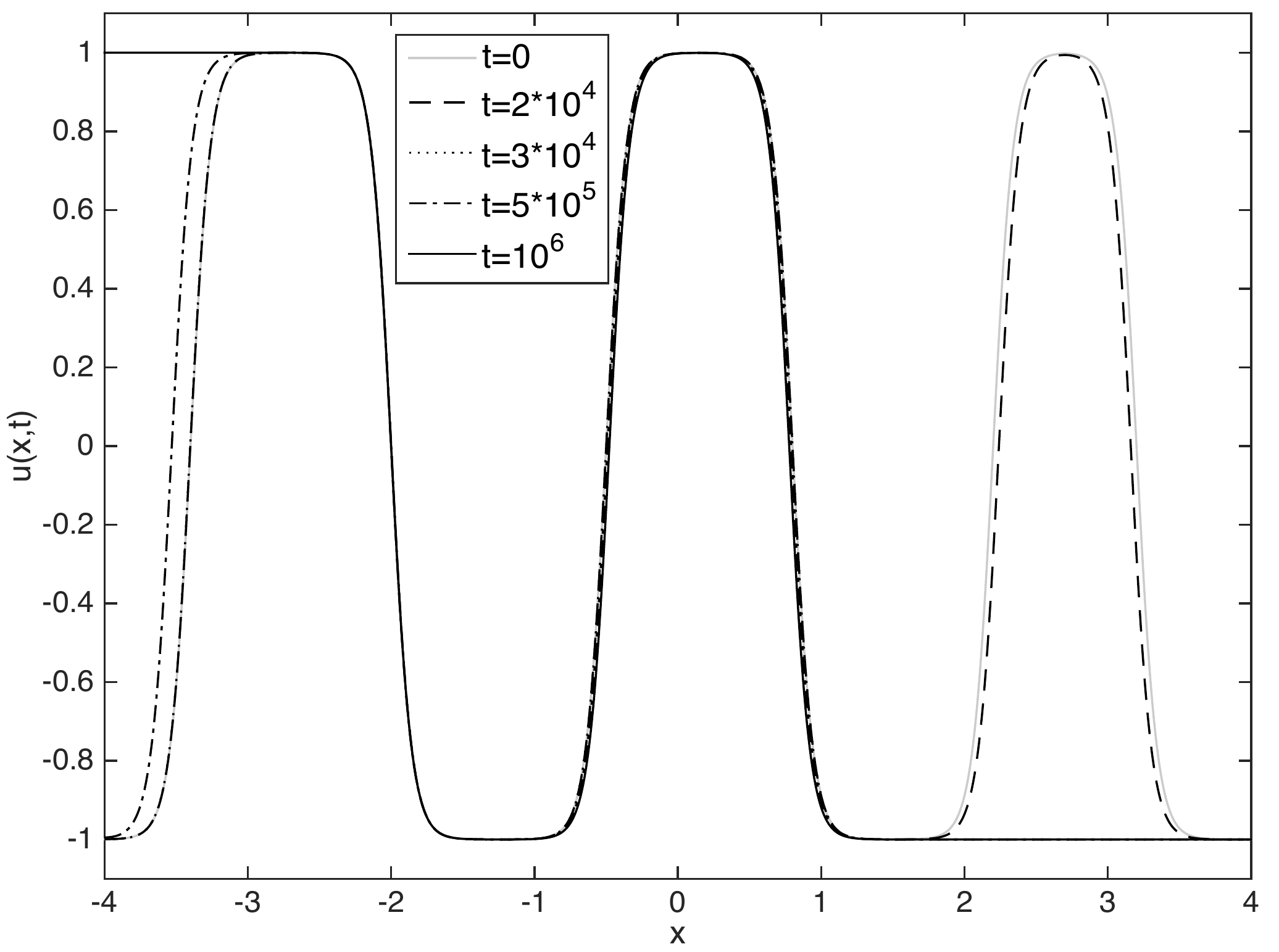}}
\subfigure[Unbounded diffusion: $Q$ given by \eqref{eq:curv-}]{\label{1right}\includegraphics[scale = 0.38]{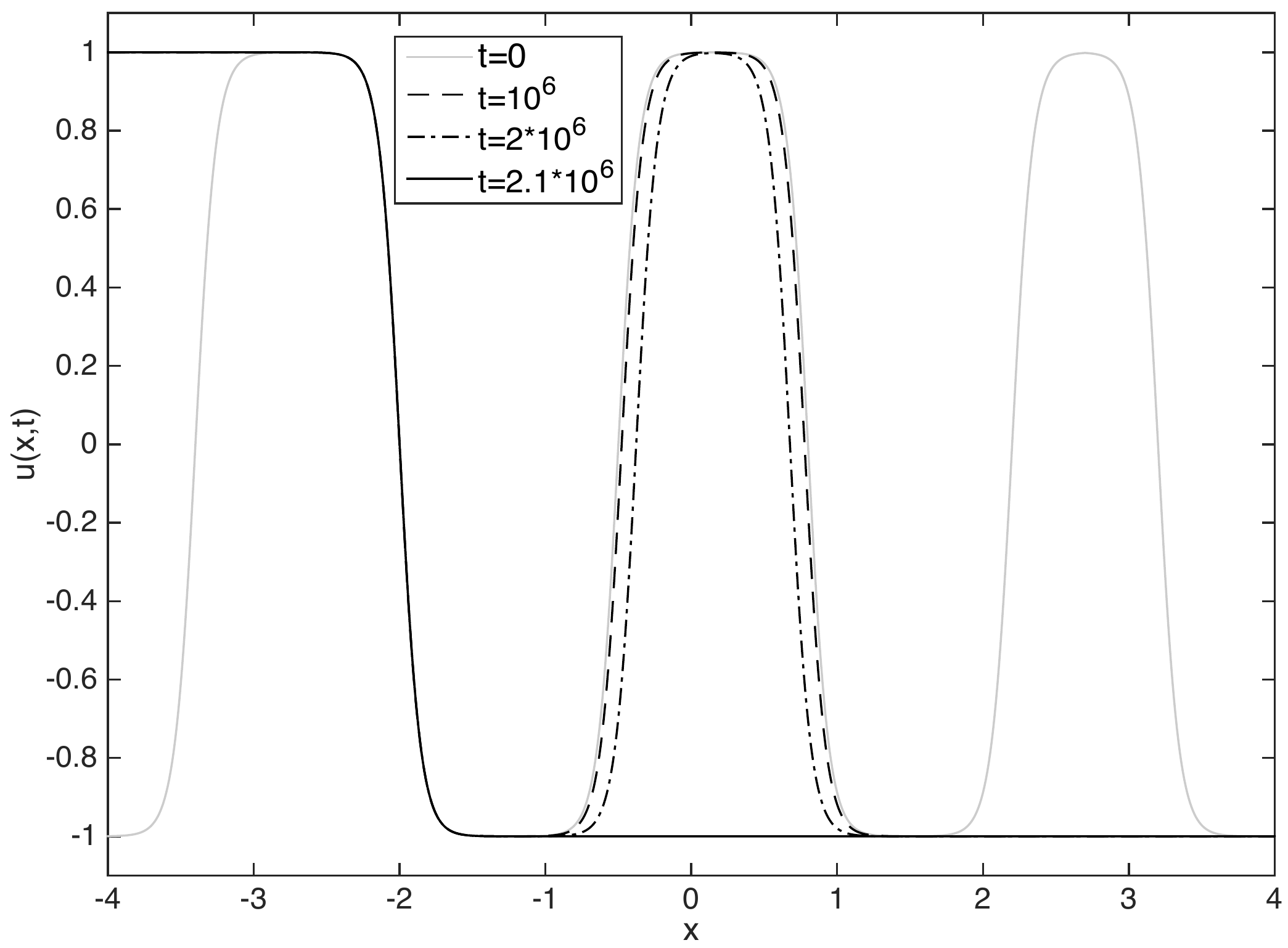}}
\caption{Numerical solutions to \eqref{eq:Q-model}-\eqref{eq:Neu}-\eqref{eq:initial} with $F(u)=\frac14(u^2-1)^2$, $\e=0.1$
and an initial profile $u_0$ with a 6-transition layer structure.
The transition points are located at $(-3.4,-2,-0.5,0.8,2.2,3.2)$.
In Figure \ref{1left} $Q$ is given by \eqref{eq:curv+}, while in Figure \ref{1right} $Q$ is determined by \eqref{eq:curv-}.}
\label{fig:6trans}
\end{figure}

In Figure \ref{1left}, we plot the solution of the IBVP \eqref{eq:Q-model}-\eqref{eq:Neu}-\eqref{eq:limit} 
with $Q$ given by \eqref{eq:curv+} for different times $t$. It can be observed that the solution maintains the same $6$-transition layer structure of the initial datum for a long time.
More precisely, the solution at time $t=2*10^4$ is very close to the initial datum (in grey), whereas at time $t=3*10^4$ the two closest transition points have already collapsed. 
Therefore, the solution maintains the same transition layer structure of the initial datum for a time $t>2*10^4$ 
(observe that the distance between the two transition points is $d=1$ and $\exp(1/\e)=\exp(10)=2.2*10^4$).
Next, at time $t=5*10^5$ another layer disappears and for $t=10^6$ the solution has a $3$-transition layer structure and it is still far away from the asymptotic equilibrium.
Finally, it is important to notice that the solution at time $t=10^6$ is almost indistinguishable from the initial datum near to the three transition points left.
As in the linear case, the two closest transition points collapse, while the other ones seem static.

Figure \ref{1right} illustrates the computed solutions in the case where the diffusion function is given by \eqref{eq:curv-}.
The $6$ transition points are located at the same positions as before and, 
since there are no important differences in the metastable dynamics of the solution until $t=10^6$, we plot the solution for later times.
We see that the transition points collapse at time $t=2.1*10^6$ and, after that, the solution has a $1$-transition layer structure.
As before, notice that the only transition point left at time $t=2.1*10^6$ is very close (indistinguishable) to the one of the solution at  time $t=0$.
Because of the homogeneous Neumann boundary conditions, we expect that the last transition point is attracted by the left boundary point 
and that the solution reaches the constant equilibrium $u \equiv -1$ after a time $t=\exp(4/\e)\approx10^{17}$. 

\subsection{Numerical experiment no. 2: initial datum without layer structure}

In the next numerical simulation, we show that metastability is a general phenomenon for the model \eqref{eq:Q-model}, since it appears even when one does not consider initial data with a layer structure.
In particular, we expect (as it happens in the linear case \eqref{eq:Al-Ca}, cf. \cite{ChenX04}) that, since $\e$ is very small, 
in the first phase of the dynamics we can neglect the term $\e^2 Q(\e^2u_x)_x$, so that the evolution of the solution is described by the reduced equation $u_t=-F'(u)$.
Hence, if the initial datum has $N$ sign changes located at some points $h_1,\dots,h_N$, 
then, after a very short time, the solution will have an $N$-transition layer structure with layers  located exactly at $h_1,\dots,h_N$. 

Let us start by considering a saturating diffusion function $Q$ given by \eqref{eq:curv+}, the same double well potential \eqref{dwpot} as in the first numerical experiment and the discontinuous initial datum
\begin{equation*}
	u_0(x)=
	\begin{cases}
		0.1, \qquad \qquad & x\in[-1,-0.05)\cup(0.05,1],\\
		-0.1, & x\in(-0.05,0.05),\\
		0, & x=\pm0.05.	
	\end{cases}
\end{equation*}
Therefore, the initial datum is a small perturbation of $u \equiv 0$, which is an unstable equilibrium point. Figure \ref{fig:discon} illustrates the computed numerical solutions to \eqref{eq:Q-model}-\eqref{eq:Neu}-\eqref{eq:initial} with $Q$ given by \eqref{eq:curv+}, the double well potential \eqref{dwpot}, $\e=0.01$
and discontinuous initial datum $u_0$, which has two jumps at $x=-0.05$ and $x=0.05$.

We see that the solution becomes smooth in a very short time and at time $t=6$ the solution has a $2$-transition layer structure
(Figure \ref{2left});
the distance $d$ between the transition points is small ($d=0.1$), but we choose $\e=0.01$ and then we observe a  metastable dynamics (Figure \ref{2right}). Indeed, the structure appears stable until time $t=8*10^4$ and after that the two transition points disappear and the solution reaches the constant equilibrium $u \equiv 1$.

\begin{figure}[hbtp]
\centering
\subfigure[$0 \leq t \leq 6$]{\label{2left}\includegraphics[scale = 0.38]{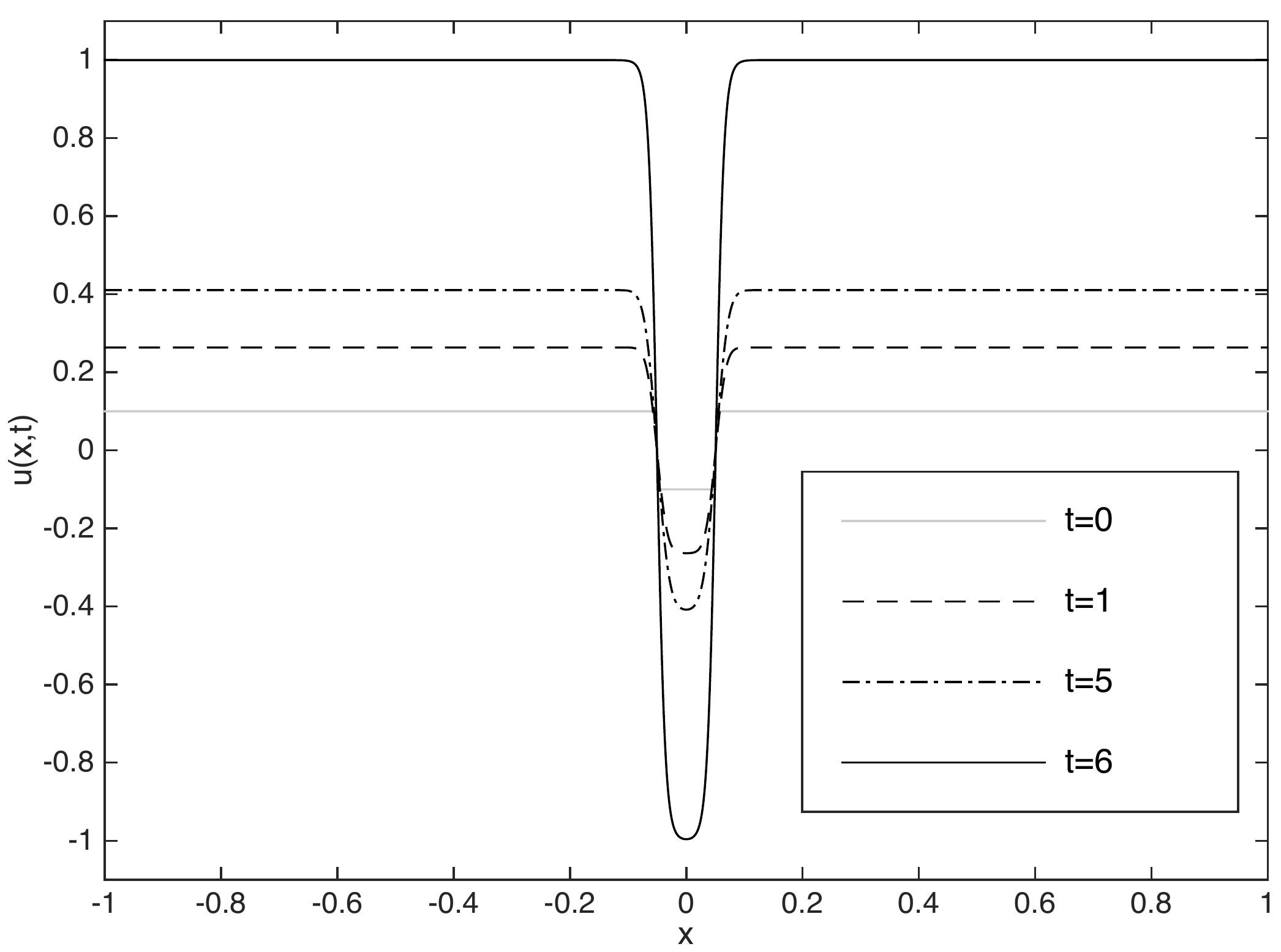}}
\subfigure[$0 \leq t \leq 9*10^4$]{\label{2right}\includegraphics[scale = 0.38]{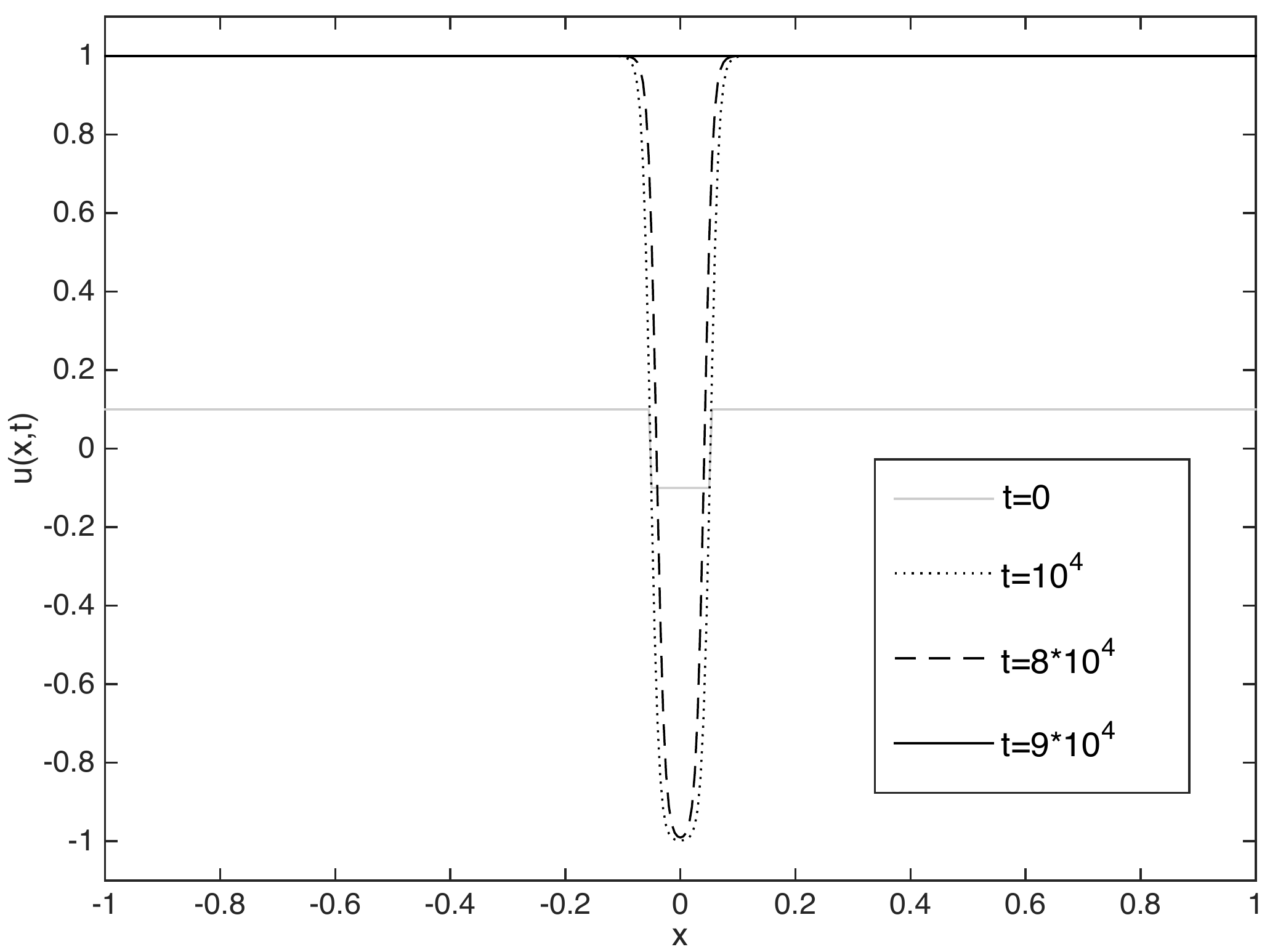}}
\caption{Numerical solutions to \eqref{eq:Q-model}-\eqref{eq:Neu}-\eqref{eq:initial} with $Q$ given by \eqref{eq:curv+}, $F(u)=\frac14(u^2-1)^2$, $\e=0.01$
and discontinuous initial datum $u_0$, which has two jumps in $-0.05$ and $0.05$.}
\label{fig:discon}
\end{figure}

The next numerical experiment shows that metastable states are attractive also in the case of the unbounded diffusion $Q$ given by \eqref{eq:curv-}. In this case, we cannot consider discontinuous initial data. Hence, we opted for a smooth perturbation of $u \equiv 0$ with a finite number of sign changes. Figure \ref{fig:0trans-} contains the plots of the computed numerical solution to \eqref{eq:Q-model}-\eqref{eq:Neu}-\eqref{eq:initial} with $Q$ given by \eqref{eq:curv-}, $F(u)=\frac14(u^2-1)^2$, $\e=0.1$
and initial datum given by 
\[
u_0(x) = \frac{1}{100}\Big[\cos\big(\frac{\pi x}{2}\big)+\sin\big(\frac{\pi x}{2}\big)\Big]. 
\]

In Figure \ref{3left} we observe the formation of the metastable state (notice that the layers are located exactly in the zeros of the initial datum). In Figure \ref{3right} we recognize the metastable dynamics:
for $t=10^4$ the solution has the same transition layer structure as in $t=10$, while when $t=10^5$ the first transition layer has already disappeared.

\begin{figure}[hbtp]
\centering
\subfigure[$0 \leq t \leq 10$]{\label{3left}\includegraphics[scale = 0.38]{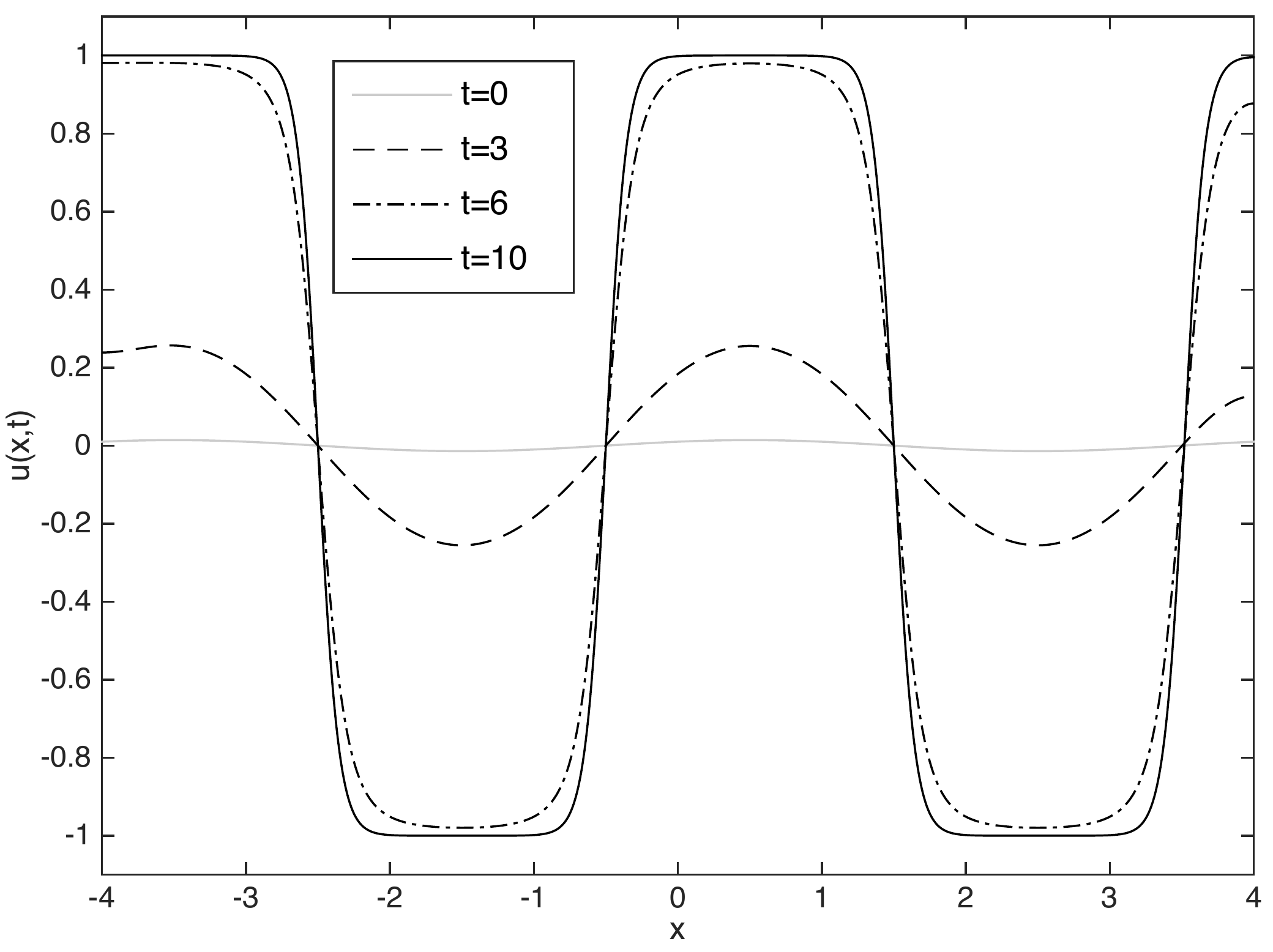}}
\subfigure[$0 \leq t \leq 10^5$]{\label{3right}\includegraphics[scale = 0.38]{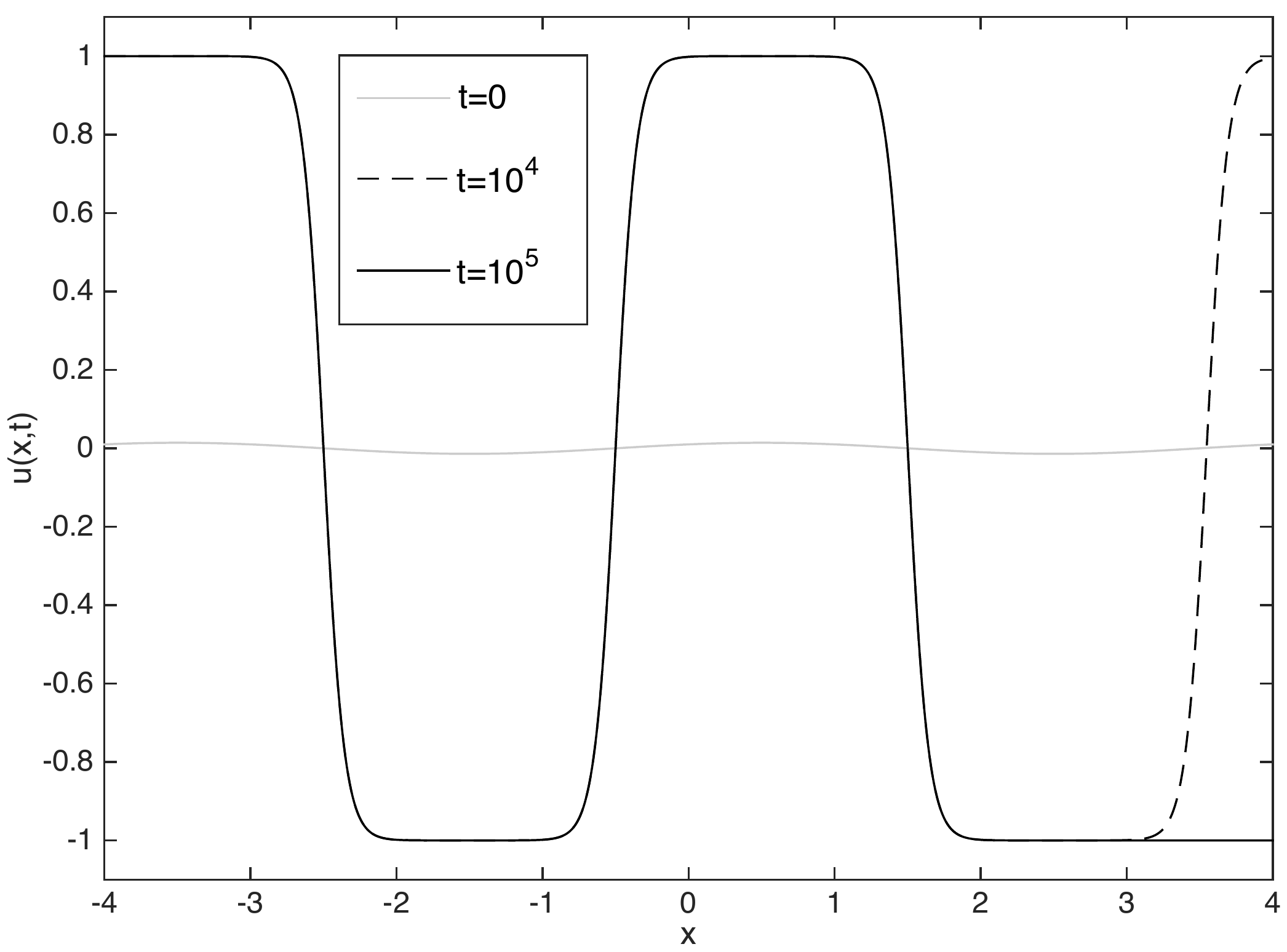}}
\caption{Numerical solutions to \eqref{eq:Q-model}-\eqref{eq:Neu}-\eqref{eq:initial} with $Q$ given by \eqref{eq:curv-}, $F(u)=\frac14(u^2-1)^2$, $\e=0.1$
and initial datum $u_0(x)=[\cos(\frac{\pi x}{2})+\sin(\frac{\pi x}{2})]/100$.}
\label{fig:0trans-}
\end{figure}

\subsection{Numerical experiment no. 3: degenerate potential}

In the previous calculations, the potential under consideration was determined by the double well function \eqref{dwpot} so that the reaction term in \eqref{eq:Q-model} is $-F'(u) = u(1-u^2)$. In these last numerical simulations, we choose the following potential,
\begin{equation}\label{eq:F-n}
	F(u)=\frac1{4n}(u^2-1)^{2n}, \qquad n>1,
\end{equation}
for which the corresponding reaction term becomes $-F'(u) = u(1-u^2)^{2n-1}$.
Notice that $F$ is a double well potential with wells of equal depth, but it is \emph{degenerate} at the wells in the sense that $F''(\pm1)=0$.
Such \emph{degenerate} case has been studied in \cite{BetSme2013} for the linear diffusion flux $Q(s)=s$ and the authors proved that the 
exponentially slow motion is replaced by an \emph{algebraic slow motion}, that is the speed of the layers
is proportional to $\e^{k}$, where the power $k$ depends on the degeneracy of the wells, for details see \cite{BetSme2013}.  
Here, we numerically study the degenerate case and show that, as in the linear case, the solutions do not exhibit exponentially slow motion.

First, we consider the case of saturating diffusion $Q$ given by \eqref{eq:curv+}, the potential \eqref{eq:F-n} with $n = 2$ and initial datum $u_0$ underlying a 6-transition layer structure, with transition points located at $h_j = -3.4,-2,-0.5,0.8,2.2$ and $3.2$, just as in the numerical experiment no. 1. 
Figure \ref{fig:deg+} contains the plots of the computed solutions to \eqref{eq:Q-model}-\eqref{eq:Neu}-\eqref{eq:initial}. 
It is to be observed that the situation drastically changes with respect to the first (non degenerate) numerical simulation (see Figure \ref{1left}); indeed, when $t=450$  the two closest transition points have already collapsed, and for $t=2000$ only one transition is left.

\begin{figure}[hbtp]
\centering
\subfigure[$0 \leq t \leq 450$]{\label{4left}\includegraphics[scale = 0.38]{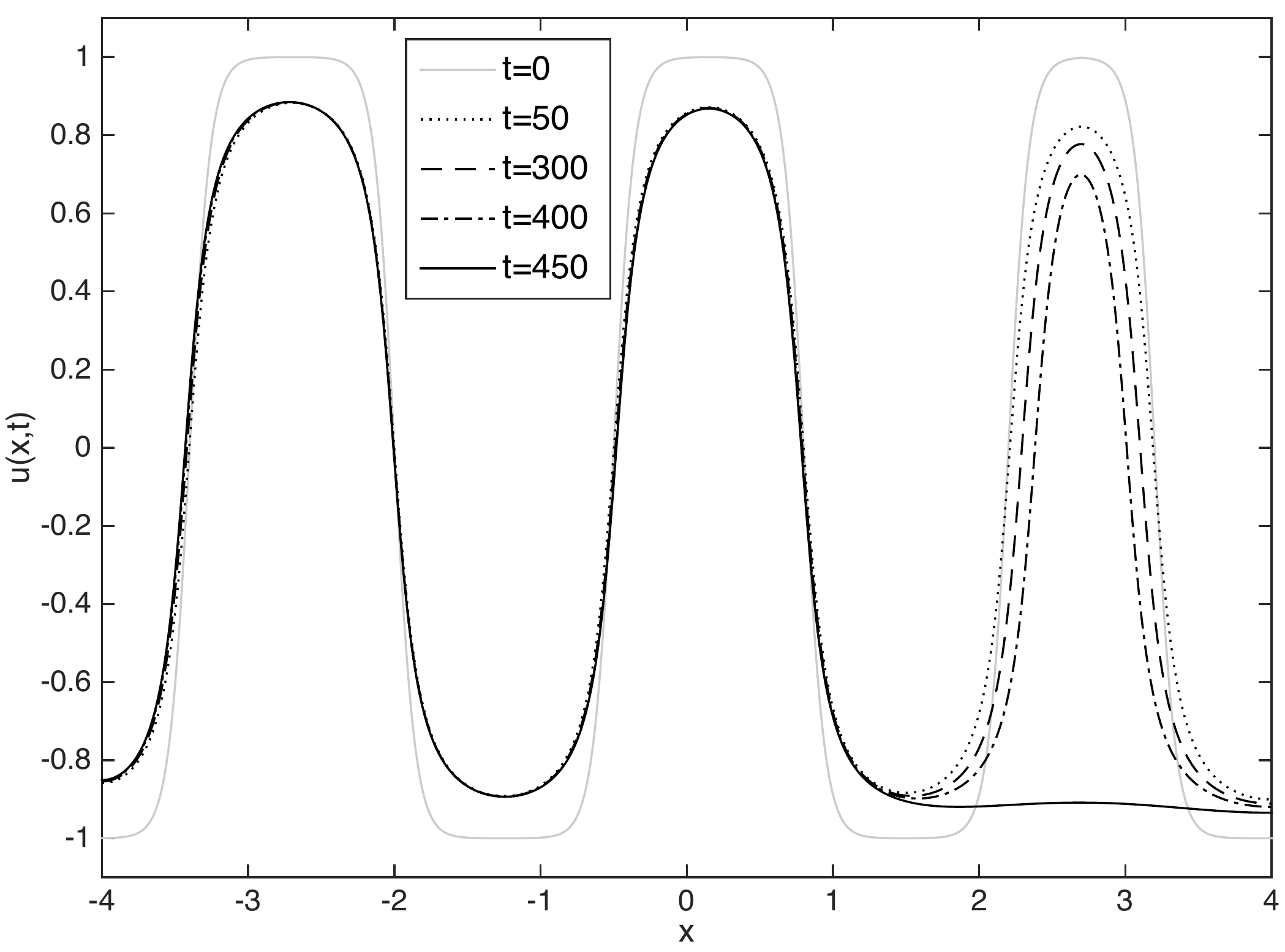}}
\subfigure[$0 \leq t \leq 2000$]{\label{4right}\includegraphics[scale = 0.38]{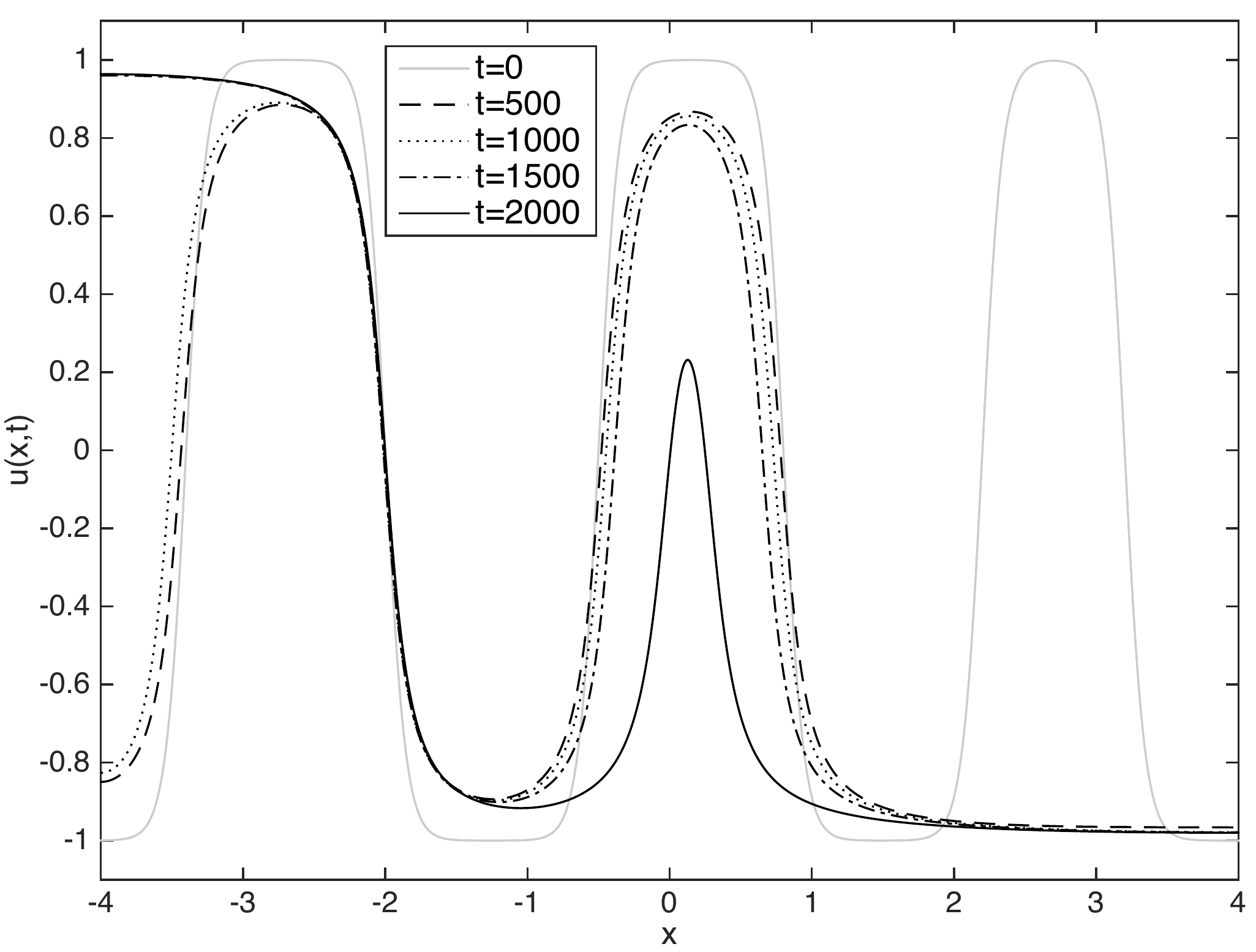}}
\caption{Numerical solutions to \eqref{eq:Q-model}-\eqref{eq:Neu}-\eqref{eq:initial} with $Q$ given by \eqref{eq:curv+}, $F(u)=\frac18(u^2-1)^4$, $\e=0.1$
and initial datum $u_0$ with a 6-transition layer structure and transition points located at $(-3.4,-2,-0.5,0.8,2.2,3.2)$, just as in the first numerical experiment (see Figure \ref{fig:6trans}).}
\label{fig:deg+}
\end{figure}

Finally, let us consider an unbounded diffusion $Q$ as in \eqref{eq:curv-}, $n=3$ in \eqref{eq:F-n} and 
the same values of $\e$ and $u_0$ as in the previous calculation. Figure \ref{fig:deg-} exhibits the results of the numerical solutions.  
Here we observe that the time of the collapse of the transition layers is even smaller than the one of Figure \ref{fig:deg+}.

\begin{figure}[hbtp]
\centering
\subfigure[$0 \leq t \leq 300$]{\label{4left}\includegraphics[scale = 0.38]{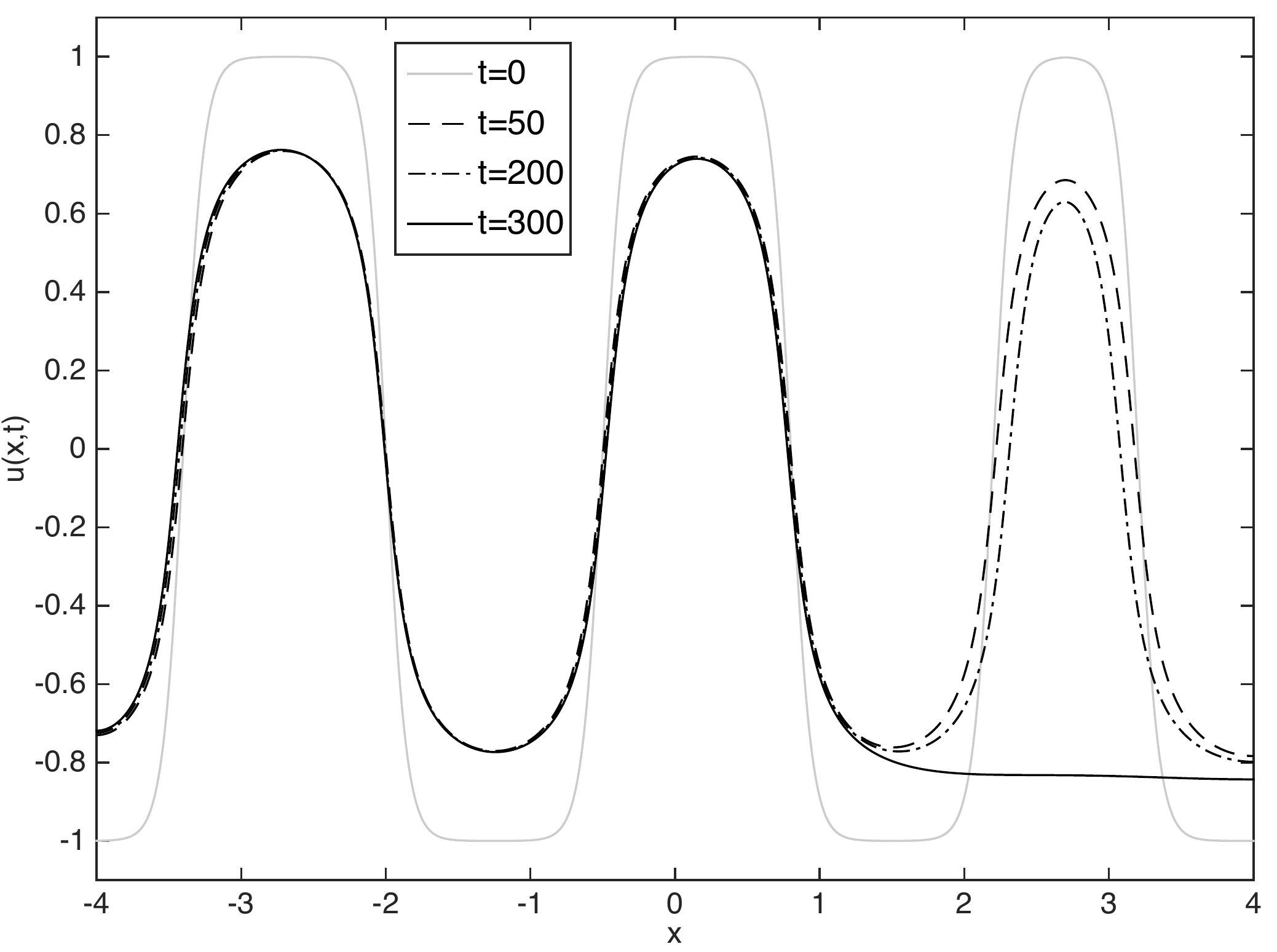}}
\subfigure[$0 \leq t \leq 1200$]{\label{4right}\includegraphics[scale = 0.38]{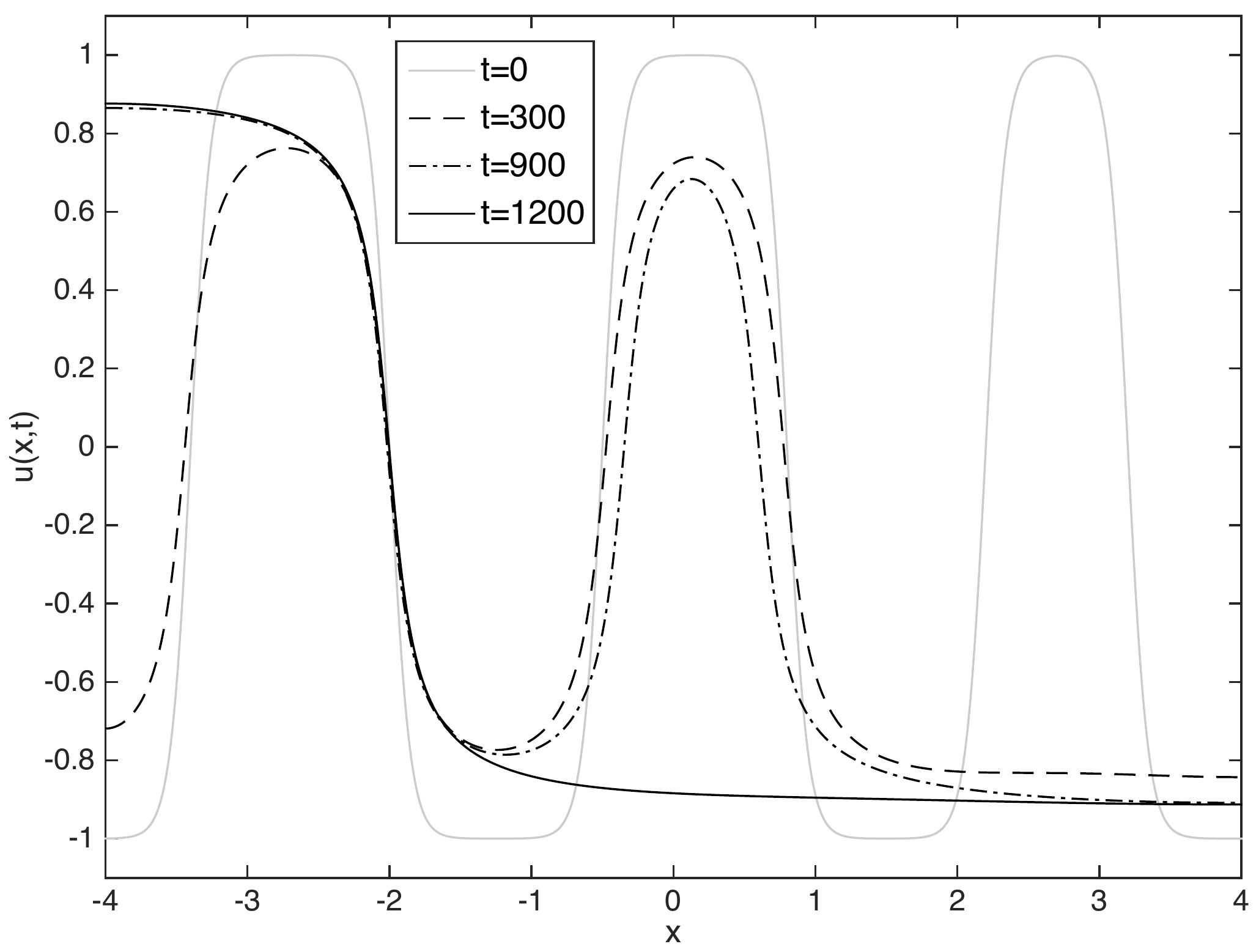}}
\caption{Numerical solutions to \eqref{eq:Q-model}-\eqref{eq:Neu}-\eqref{eq:initial} with $Q$ given by \eqref{eq:curv-}, $F(u)=\frac1{12}(u^2-1)^6$, $\e=0.1$
and initial datum $u_0$ with a 6-transition layer structure and transition points located at $(-3.4,-2,-0.5,0.8,2.2,3.2)$.}
\label{fig:deg-}
\end{figure}


In conclusion, the computations presented here constitute numerical evidence that assumption $F''(\pm1)>0$ is crucial in order to observe a metastable dynamics for equation \eqref{eq:Q-model}.

\section*{Acknowledgements}

The work of RGP was partially supported by DGAPA-UNAM, program PAPIIT, grant IN-100318.

\bibliography{riferimenti}

\bibliographystyle{newstyle}

\end{document}